\documentclass[12pt, a4paper]{amsart}
\usepackage{amsmath}
\usepackage{geometry,amsthm,graphics,tabularx,amssymb,shapepar}
\usepackage{amscd}
\usepackage{bm}
\usepackage[all]{xypic}

\newcommand{\fg}{\mathfrak g}
\newcommand{\fh}{\mathfrak h}
\newcommand{\dg}{\dot{\mathfrak g}}
\newcommand{\dfh}{\dot{\mathfrak h}}

\newcommand{\al}{\alpha}

\newcommand{\wt}{\widetilde}
\newcommand{\wh}{\widehat}

\newcommand{\detmn}{\det{\bm{m}\choose \bm{n}}}
\newcommand{\ot}{\otimes}

\newcommand{\CR}{\mathcal{R}}
\newcommand{\CG}{\mathcal{G}}

\newcommand{\CH}{\mathcal{H}}

\newcommand{\CS}{\mathcal{S}}
\newcommand{\CE}{\mathcal{E}}

\newcommand{\C}{\mathbb{C}}
\newcommand{\R}{\mathbb R}
\newcommand{\N}{\mathbb N}
\newcommand{\Z}{\mathbb Z}
\newcommand{\rd}{\mathrm{d}}
\newcommand{\rk}{\mathrm{k}}
\newcommand{\rh}{\mathrm{h}}
\newcommand{\re}{\mathrm{e}}
\newcommand{\rf}{\mathrm{f}}

\newcommand{\ba}{\begin {eqnarray}}
\newcommand{\ea}{\end {eqnarray}}
\newcommand{\baa}{\begin {eqnarray*}}
\newcommand{\eaa}{\end {eqnarray*}}
\newcommand{\be}{\begin {equation}}
\newcommand{\ee}{\end {equation}}
\newcommand{\bee}{\begin {equation*}}
\newcommand{\eee}{\end {equation*}}

\newcommand{\U}{\mathcal{U}}

\newcommand{\te}[1]{\textnormal{{#1}}}

\theoremstyle{Theorem}

\theoremstyle{Theorem}

\newtheorem{thm}{Theorem}[section]
\newtheorem{cort}[thm]{Corollary}
\newtheorem{lemt}[thm]{Lemma}
\newtheorem{prpt}[thm]{Proposition}

\theoremstyle{Theorem}

\theoremstyle{Theorem}

\theoremstyle{Plain}

\theoremstyle{Definition}

\newtheorem{dfnt}[thm]{Definition}

\def\({\left(}

\def\){\right)}

\def \<{{\langle}}
\def \>{{\rangle}}

\allowdisplaybreaks

\numberwithin{equation}{section}
\title[Repn of EALAs]{Classification of integrable representations for toroidal extended affine Lie algebras}

\author{Fulin Chen$^1$}
\thanks{{1. Partially supported by China NSF grants (No.11501478)}}
\address{School of Mathematics Science, Xiamen University,
 Xiamen, China 361005} \email{chenf@xmu.edu.cn }
\author{Zhiqiang Li}
\address{School of Mathematics Science, Xiamen University,
 Xiamen, China 361005} \email{lzq06031212@sina.com}
 \author{Shaobin Tan$^2$}
 \thanks{{2. Partially supported by China NSF grants (Nos.11471268, 11531004)}}
 \address{School of Mathematics Science, Xiamen University,
 Xiamen, China 361005} \email{tans@xmu.edu.cn}


\begin{document}

\begin{abstract}
 In this paper, we classify the irreducible integrable modules 
   with finite dimensional weight spaces and non-trivial $\wt\fg_c$-action 
 for the nullity $2$ toroidal extended affine Lie algebra $\wt\fg$,
  where $\wt\fg_c$ is the core of $\wt\fg$.

\end{abstract}
\maketitle

\section{Introduction}
In this paper we study the integrable representation for the nullity $2$ toroidal extended affine Lie algebras (EALAs for short).
The notion of EALAs  was first introduced by Hoegh-Krohn and Torresani in \cite{H-KT} under the name of
quasi-simple Lie algebras. Since then the structure theory of EALAs has been intensively studied for the past twenty years
 (see \cite{AABGP,BGK,N}
and the references therein).
 An EALA $E$ is by definition a complex Lie algebra, together with a non-zero finite dimensional
ad-diagonalizable subalgebra $H$  and a non-degenerate invariant symmetric bilinear form $(\cdot|\cdot)$,
that satisfies
a list of natural axioms.
A crucial consequence of these axioms is that
the form $(\cdot|\cdot)$ on $E$ induces a semi-positive bilinear form on the $\R$-span of the root system of $E$ relative to $H$.
    Then the root system of $E$ divides into a disjoint union of the sets of isotropic and non-isotropic roots.
    The subalgebra $E_c$ of $E$ generated by non-isotropic root vectors is called the core of $E$, which is in fact an ideal of $E$.
An $E$-module is said to be integrable if it admits a weight space decomposition
relative to $H$  and all non-isotropic root vectors in $E$ act locally nilpotent on it.
We remark that the integrability of a weight $E$-module $V$   depends only  on the action of $E_c$.
In particular, if $E_c.V=0$, then $V$ is automatically integrable.

 One of the axioms for EALAs requires that the rank of the
group generated by the isotropic roots is finite, and this rank is called the nullity of EALAs.
The nullity $0$ EALAs are nothing but the finite dimensional simple Lie algebras,
and the affine Kac-Moody algebras are precisely the nullity $1$  EALAs \cite{ABGP}.
It is well-known that in the study of the representation theory for finite or affine   Kac-Moody algebras, one of the main
ingredients is the classification of irreducible integrable representations.
The irreducible integrable modules (with finite  dimensional weight spaces) for the finite dimensional simple Lie
algebras are precisely
the  finite dimensional irreducible modules.
And,  such modules for the affine Kac-Moody algebras are the highest weight integrable modules,
their dual or the loop modules \cite{C,CP1,CP2}.
Similar to the nullity $0$ and nullity $1$ cases, one of the fundamental problems in the representation theory
of extended affine Lie algebras with higher nullities is to classify their irreducible integrable modules.

Toroidal extended affine Lie algebras  are  a class of
important EALAs that provide examples with arbitrary  nullities.
The construction of toroidal EALAs parallels with the untwisted affine Lie algebras.
Let $\dg$ be a finite dimensional simple Lie algebras, 
let $\mathcal R=\C[t_0^{\pm 1},\cdots,t_{N-1}^{\pm 1}]$ be the ring of Laurent polynomials
 in $N$  variables and let
  $\mathcal S$ be the
subspace of divergence zero derivations on $\mathcal R$ (which is also called the set of skew-derivations \cite{BGK}).
The nullity $N$ toroidal EALA
$\wt\fg=\wt\fg_c \oplus \CS$
is defined
by adding  the space $\mathcal S$ (possibly twisted
with a $2$-cocycle)
to the
 universal central extension
$\wt\fg_c=(\mathcal R\ot \dg)\oplus \mathcal K$ of the   multi-loop Lie algebra $\mathcal R\ot \dg$.
We recall that $\wt\fg_c$ is the core of $\wt\fg$ and is often called the toroidal Lie algebra.
One may see Section 2 or \cite[Section 1]{B} for details.

 Let $\mathcal D$ be the space of derivations over $\mathcal R$.
As a left $\mathcal R$-module, $\mathcal D$ has a basis $\rd_0,\cdots, \rd_{N-1}$, where $\rd_i=t_i\frac{\partial}{\partial t_i}$.
Besides toroidal EALAs, there are some other interesting extensions of the toroidal Lie algebra $\wt\fg_c$  that are obtained by
adding various subspace of $\mathcal D$. For examples, one can add the subspaces $\mathcal D_0=\C\rd_0\oplus \cdots \oplus\C\rd_{N-1}$,
$\mathcal D^*=\C\rd_0\oplus \CR \rd_1\oplus \cdots \oplus\CR\rd_{N-1}$ of $\mathcal D$, or the space $\mathcal D$ itself.
The irreducible integrable representations of the Lie algebras
 $\wt\fg_c\oplus \mathcal D_0$, $\wt\fg_c\oplus \mathcal D^*$ and  $\wt\fg_c\oplus \mathcal D$  with finite dimensional
weight spaces have been studied respectively in \cite{E1,E2},\cite{JM} and \cite{EJ}.
This paper is aimed to classify the irreducible integrable modules for the nullity $2$ toroidal EALAs.

We denote by $\wt\CS$  the quotient algebra of $\wt\fg$ obtained by modulo its core $\wt\fg_c$, which is isomorphic to the
Lie algebra of divergence zero derivations. When $N=2$, the derived subalgebra of $\wt\CS$ is often called
the (centerless) Virasoro-like algebra.
As we mention before, any weight $\wt\fg$-module with trivial $\wt\fg_c$-action is integrable.
Thus the category of integrable $\wt\fg$-modules with trivial $\wt\fg_c$-action
is equivalent to the category of weight $\wt\CS$-modules.
In recent years, the classification problem on the irreducible weight $\wt\CS$-modules with finite dimensional weight spaces
 has been studied by several authors, see \cite{BT,GL,JL,LS,LT} for details.

 In this paper, we give a complete classification of irreducible integrable representations for the nullity $2$ toroidal EALA $\wt\fg$ with
 finite dimensional weight spaces and non-trivial $\wt\fg_c$-action.
We remark that the classification  of the irreducible integrable modules for the nullity $N$ toroidal EALAs with $N\ge 3$ is different to that of the $N=2$ case, and will be given in a subsequent paper. So, from now on, in the rest of the paper we always assume that $N=2$.

To be more precisely about our result, 
we say that  an irreducible integrable $\wt\fg$-module
 $V$ has zero (resp. non-zero) central charge if  the central elements $\rk_0, \rk_1$ of $\wt\fg$
 act trivially (resp. non-trivially) on $V$.
In \cite{EJ}, by using the method developed in \cite{JM},  the classification of  irreducible
 integrable modules for
 $\wt\fg_c\oplus \mathcal D$ with zero central charge is reduced to the classification of
 irreducible modules for the semi-product algebra
$\mathcal R\rtimes \mathcal D$ with certain conditions.
As in \cite{EJ}, in the present work
 we reduce the classification
  problem for integrable $\wt\fg$-modules with zero central charge
  to the classification problem for  certain $\mathcal R\rtimes \wt\CS$-modules.
Such  $\mathcal R\rtimes \wt\CS$-modules have been classified in \cite{JL} (see also \cite{GL,BT}).
We emphasise that  Jiang-Meng's reduction method can not be applied directly to the toroidal EALAs,
as almost all lemmas used in this reduction need certain distinguished derivations that do not belong to $\CS$.
See \cite[Lemma 3.1, 3.2, 3.4, 3.9]{JM} for details.
To make Jiang-Meng's reduction work for the toroidal EALAs, we use rather different tricks to fulfill each step.
We first prove in Subsection 4.2 that the whole space $\mathcal K$ acts trivially on any irreducible integrable $\wt\fg$-module with
zero central charge.
The proof of this  assertion is based on a result of Chari-Pressely \cite{CP2} 
on integrable modules for extended loop algebras (see Lemma \ref{eloop1}).
Then, by using the technique developed in \cite{GL}, we give a  classification of the irreducible uniformly bounded
 $(\mathcal R\ot \dfh)\rtimes \wt\CS$-modules in Subsection 4.3, where $\dfh$ stands for the Cartan subalgebra of $\dg$.
 Finally, by putting the above results together, in Subsection 4.4 we  classify the irreducible integrable modules for $\wt\fg$
 in Theorem \ref{mainhw1} for the zero central charge case.
 A realization of such irreducible integrable $\wt\fg$-modules is also presented in Subsection 4.4.

In \cite{E2}, after a change of coordinates, 
the classification of irreducible integrable $\wt\fg_c\oplus \mathcal D_0$-modules with non-zero central charge is 
reduced to determine certain linear functionals on a distinguished Heisenberg algebra.
As done in \cite{E2},  after a change of coordinates, 
 we show that 
each irreducible integrable $\wt\fg$-module with non-zero central charge is a highest weight module associated
to certain linear functionals on a Heisenberg algebra $\wh\CH$ (see \eqref{hatch}).
 The sufficient and  necessary condition for such a highest weight module to be integrable and has
 finite dimensional weight spaces is given in Theorem \ref{mainhw2} and proved in Section 5.
 It turns out that the characterization of these linear functionals on $\wh\CH$ is more 
 subtle than that given in \cite{E2}.
 In particular, the notion of exp-polynomial function is needed in the toroidal EALAs case. One may see Subsection 3.2 for details. 
As a loop generalization of Billig's work \cite{B},
 a class of  irreducible integrable highest weight $\wt\fg$-modules with non-zero central charge were constructed in \cite{CLT}.
 This construction  plays a key role in our proof of Theorem \ref{mainhw2}.

We now  outline the structure of the paper.
In Section 2, we review some basics on nullity $2$ toroidal EALAs and their integrable modules.
 The main result, Theorem \ref{main}, of our paper is stated and proved in Section 3.
The proof of Theorem \ref{main} is based on two classification results (Theorem \ref{mainhw1} and Theorem \ref{mainhw2}) on certain
irreducible integrable highest weight $\wt\fg$-modules.
Section 4 and Section 5 are respectively devoted to the proofs of Theorem \ref{mainhw1} and Theorem \ref{mainhw2}.

We denote the sets of integers, non-negative integers, complex numbers, nonzero complex numbers and rational numbers by $\Z$, $\N$, $\C$, $\C^\times$ and $\mathbb{Q}$, respectively.

\section{Nullity $2$ toroidal extended affine Lie algebras}
In this section we collect some basic results about  nullity $2$ toroidal extended affine Lie algebras and their integrable representations, 
which will be used later on.

\subsection{The toroidal EALA $\wt\fg$} In this subsection we review the definition of the torodial EALA $\wt\fg$.

Let $\dg$ be a finite dimensional complex simple Lie algebra and
let $\dfh$ be a Cartan subalgebra of it. We denote by $\ell=\dim \dfh$  the rank of $\dg$.
Let $\dot{\Delta}$ (containing $0$) the root system of $\dg$ relative to the Cartan subalgebra $\dfh$.
Let $\<\cdot,\cdot\>$ be a non-degenerate invariant symmetric bilinear form on $\dg$,  and we will usually identify  $\dfh^\ast$ with $\dfh$ by means of it. We further assume that the invariant form  $\<\cdot ,\cdot\>$ is normalized so that the square length of long roots equal to $2$.
For each $\al\in \dot\Delta\setminus\{0\}$, let $\al^\vee=\frac{2\al}{\< \al,\al\>}$ indicate the corresponding coroot of $\al$.

Let $\mathcal R=\C[t_0^{\pm 1}, t_1^{\pm 1}]$ be the Laurent polynomial ring in two variables.
For $\bm{m}=(m_0,m_1)\in \Z^2$, we  write
\[t^{\bm{m}}=t_0^{m_0}t_1^{m_1}.\]
Denote by $\Omega_\mathcal R^1$ the space of K\"{a}hler differentials on $\mathcal R$.
As a left $\mathcal R$-module, $\Omega_\mathcal R^1$ has a basis $\rk_0, \rk_1$, where $\rk_i$, $i=0,1$, stands for the $1$-forms $t_i^{-1}\rd t_i$.
Form the quotient space $\mathcal K=\Omega^1_\mathcal R/\rd(\Omega_{\mathcal R}^1)$, where
\[\rd(\Omega_{\mathcal R}^1)=\text{Span}_{\C}\{m_0t^{\bm{m}}\rk_0+m_1t^{\bm{m}}\rk_1\mid \bm{m}=(m_0,m_1)\in \Z^2\}\]
is the space of exact $1$-forms in $\Omega_\mathcal R^1$.
The universal central extension of the loop algebra $\mathcal R\ot \dg$, called $2$-toroidal Lie algebra, can be realized as
\[\wt\fg_c=\(\mathcal R\ot \dg\)\oplus \mathcal K,\]
and its Lie bracket is given by (see \cite{MRY},\cite{EM})
\[ [t^{\bm{m}}\ot x, t^{\bm{n}}\ot y]=t^{\bm{m}+\bm{n}}\ot [x,y]+\< x,y\> \sum_{i=0,1} m_it^{\bm{m}+\bm{n}}\rk_i,\]
where $x,y\in \dg$, $\bm{m}=(m_0,m_1), \bm{n}=(n_0,n_1)\in \Z^2$, and $\mathcal K$ is central.

We denote by $\mathcal D=\te{Der}(\mathcal R)$ the space of derivations over $\mathcal R$.
As a left $\mathcal R$-module, $\te{Der}(\mathcal R)$ has a basis $\rd_0, \rd_1$, where $\rd_i=t_i\frac{\partial}{\partial t_i}$, $i=0,1$.
Let $\mathcal S$ be the subspace of $\mathcal D$ consisting of
skew derivations (also called divergence zero derivations). So, by definition, it is the subspace of $\mathcal D$
spanned by the degree zero derivations $\rd_0, \rd_1$ and the following skew derivations
\[ \rd_{\bm{m}}=m_0 t^{\bm{m}}\rd_1-m_1 t^{\bm{m}}\rd_0,\]
for $ \bm{m}\in \Z^2$.

Let $\mu$ be a fixed complex number.
We define an extension $\wt\fg(\mu)$ of $\wt\fg_c$ by the skew derivations as follows.
The underlying vector space of $\wt\fg(\mu)$ is
\[\widetilde{\fg}(\mu)=\wt\fg_c\oplus \mathcal S=\(\mathcal R\ot \dg\)\oplus \mathcal K\oplus \mathcal S,\]
and the remaining bracket relations on $\wt\fg$ are given by
\begin{eqnarray*}
&&[\rd_{\bm{m}}, t^{\bm{n}}\ot x]=\detmn t^{\bm{m}+\bm{n}}\ot x,\quad [\rd_i,t^{\bm{n}}\ot x]=n_i\, t^{\bm{n}}\ot x, \\
&&[\rd_{\bm{m}}, t^{\bm{n}}\rk_j]=\detmn t^{\bm{m}+\bm{n}}\rk_j+ (m_0\delta_{j,1}-m_1\delta_{j,0})\sum_{a=0,1} m_a\,t^{\bm{m}+\bm{n}}\rk_a,\\
&&[\rd_i, t^{\bm{n}}\rk_j]=n_i\, t^{\bm{n}}\rk_j,\quad [\rd_i,\rd_{\bm{m}}]=m_i\,\rd_{\bm{m}},\quad [\rd_0,\rd_1]=0,\\
&&[\rd_{\bm{m}},\rd_{\bm{n}}]=\detmn\rd_{\bm{m}+\bm{n}}+\mu\(\detmn\)^{2}\sum_{a=0,1} m_a\, t^{\bm{m}+\bm{n}}\rk_a,
\end{eqnarray*}
for $x\in \dg$, $\bm{m},\bm{n}\in \Z^2$ and $i,j=0,1$, where
\[\detmn=m_0n_1-n_0m_1.\]
The resulting Lie algebra $\wt\fg(\mu)$ is an extended affine Lie algebra with nullity $2$, and is often called
the nullity $2$ toroidal EALA.
For the sake of simplicity, in what follows we will  write $\wt\fg=\wt\fg(\mu)$ for convenience.
Note that the Cartan subalgebra of $\widetilde{\fg}$ is
given by
\[\widetilde{\fh}=\dfh\oplus \C\rk_0\oplus \C\rk_1 \oplus \C\rd_0 \oplus \C\rd_1,\]
and the non-degenerate invariant symmetric bilinear form $\<\cdot , \cdot\>$ on $\widetilde{\fg}$ can be determined as follows
\begin{eqnarray*}
&&\< t^{\bm{m}}\ot x, t^{\bm{n}}\ot y\>=\delta_{m_0+n_0,0}\delta_{m_1+n_1,0}
\< x, y\>,\ \< \rd_i, \rk_j \>=\delta_{i,j},\\
&&\<  \rd_{\bm{m}},  t_0^{n_0}t_1^{n_1}\rk_j\>=
(m_0\delta_{1,j}-m_1\delta_{0,j})\delta_{m_0+n_0,0}\delta_{m_1+n_1,0},
\end{eqnarray*}
and the others are trivial.

For each  $\al\in \dot{\Delta}$, we extend it to  a linear functional in $\wt\fh^*$ so that
\[\al(\rk_i)=0=\al(\rd_i),\]
for $i=0,1$, and define linear functionals $\delta_i$, $i=0,1$ on $\wt\fh$  by setting
\begin{align*}&\delta_i(\dfh)=0,\quad \delta_i(\rk_j)=0,\quad \delta_i(\rd_j)=\delta_{i,j},
\end{align*}
for $j=0,1$. Then the set
\[\wt\Delta=\{\al+m_0\delta_0+m_1\delta_1\mid \al\in \dot\Delta,\ m_0,m_1\in \Z\}\]
is the root system of $\wt\fg$ with respect to the Cartan subalgebra $\wt\fh$. And we have the following root space decomposition of $\fg$
\[\wt\fg=\bigoplus_{\al\in \wt\Delta}\wt\fg_\al,\]
where
\[\wt\fg_\al=\{x\in \wt\fg\mid [h,x]=\al(h)x,\  h \in \wt\fh\}.\]

\subsection{Integrable representations of $\wt\fg$}
In this subsection we introduce the definition of integrable modules for the extended affine Lie algebra $\wt\fg$.

We say that a $\wt\fg$-module $V$ is a weight module if it admits a weight space
decomposition
\[V=\oplus_{\lambda\in \wt\fh^*} V_\lambda\] relative to its Cartan subalgebra $\wt\fh$, where
\[V_\lambda=\{v\in V\mid h.v=\lambda(h)v,\ \forall h\in \wt\fh\}.\]
Note  that the restriction of $\<\cdot,\cdot\>$ on $\wt\fh$ is non-degenerate.
Thus, by duality, it induces a non-degenerate bilinear form on $\wt\fh^*$.
We denote by
\[\wt\Delta^\times=\{\al\in \wt\Delta\mid \<\al,\al\>\ne 0\}\quad \te{and}\quad
\wt\Delta^0=\{\al\in \wt\Delta\mid \<\al,\al\>= 0\}\]
the sets of non-isotropic and isotropic roots in $\wt\Delta$, respectively.


\begin{dfnt}
A $\wt\fg$-module $V$ is said to be integrable if it is a weight module  and for every $\al\in \wt\Delta^\times$,
$\wt\fg_\al$ acts locally nilpotent on $V$.
\end{dfnt}

 We say that a $\wt\fg$-module $V$ has  central charge $(c_0,c_1)\in \C^2$ if  $\rk_i.v=c_iv$ for $v\in V$, $i=0,1$.
Furthermore, we say that $V$ has non-zero (reps. zero) central charge  if  $(c_0,c_1)\ne (0,0)$ (resp. $=(0,0)$).
The following is obvious.
\begin{lemt} Let $V$ be an irreducible integrable $\wt\fg$-module. Then $V$ has central charge $(c_0,c_1)$ for some $(c_0,c_1)\in \Z^2$.
\end{lemt}

We recall that $\wt\fg_c$ is the core of $\wt\fg$, namely, the subalgebra of $\wt\fg$ generated by non-isotropic root vectors.
As in the Introduction, let $\wt{\mathcal{S}}$ be the quotient algebra of $\wt\fg$ obtained by modulo its core   $\wt\fg_c$, and let
 $\pi:\wt\fg\rightarrow \wt\CS$ be the quotient map. Set
 \[\pi(\rd_i)=\rd_i,\quad  \pi(\rd_{\bm{m}})=\rd(\bm{m}),\]
 for $ i=1,2$ and $\bm{m}\in \Z^2$.
Then the elements $\rd_i, i=0,1$ and $\rd(\bm{m})$, $\bm{m}\in \Z^2\setminus\{\bm{0}\}$ $(\bm{0}=(0,0)$) form a basis of $\wt{\mathcal{S}}$, and the Lie relations are given by
\begin{align}\label{wtcs}[\rd_i,\rd_j]=0,\quad [\rd_i,\rd(\bm{m})]=m_i\rd(\bm{m}),\quad [\rd(\bm{m}),\rd(\bm{n})]=\detmn \rd(\bm{m}+\bm{n}),\end{align}
where $i,j=0,1$, $\bm{m},\bm{n}\in \Z^2$.
The derived subalgebra of $\wt{\mathcal{S}}$ is often called the centerless Virasoro-like algebra (see \cite{LT} and the references therein).
An $\wt\CS$-module is called a weight module if it admits a weight space decomposition relative to $\C\rd_0\oplus \C\rd_1$,
and a weight $\wt\CS$-module is called quasi-finite if
its weight spaces are all finite dimensional.

We denote by $\vartheta_{\mathrm{fin}}$ the category of integrable $\wt\fg$-modules
with finite dimensional weight spaces, and denote by $\vartheta_{\mathrm{fin}}^\times$ and $\vartheta_{\mathrm{fin}}^0$ respectively the subcategories of $\vartheta_{\mathrm{fin}}$ whose objects
have non-trivial and trivial $\wt\fg_c$-action.
Let $V$ be a weight $\wt\fg$-module such that $\wt\fg_c.V=0$. Then it is integrable and admits a natural  weight $\wt\CS$-module
structure.
Conversely, via the quotient map $\pi$, any weight $\wt\CS$-module $V$ is an integrable $\wt\fg$-module  such that $\wt\fg_c.V=0$.
In particular, this gives that the category $\vartheta_{\mathrm{fin}}^0$ is isomorphic to the category of quasi-finite weight $\wt\CS$-modules.

In recent years, the classification problem on irreducible quasi-finite weight $\wt\CS$-modules
has been investigated in several papers, see \cite{BT,GL,JL,LT,LS} for details.
As mentioned in the Introduction, the main goal of this paper is to classify the irreducible objects in the category  $\vartheta_{\mathrm{fin}}^\times$.
In particular, we present a classification of the irreducible objects in $\vartheta_{\mathrm{fin}}$ with non-zero central charges.

\section{Classification of irreducible integrable $\wt\fg$-modules}
In this section we state the main results of the paper.
\subsection{Integrable highest weight $\wt\fg$-modules of type I}
In this subsection we introduce a notion of highest weight module for the algebra $\wt\fg$,
and state the sufficient and necessary condition for  such an irreducible highest  weight $\wt\fg$-module belongs to the category
$\vartheta_{\mathrm{fin}}^\times$.

Let  $\dot\Pi=\{\al_1,\cdots,\al_\ell\}$ be a fixed simple root system of $\dg$. Denote by $\dot\Delta_+$ and $\dot\Delta_-$ respectively the corresponding positive and negative root systems of $\dg$. Then we have a decomposition
\[\wt\Delta=\wt\Delta^+\cup \wt\Delta^0\cup \wt\Delta^-\]
of $\wt\Delta$, where
\[\wt\Delta^\pm=\{\al+m_0\delta_0+m_1\delta_1\mid \al\in \dot\Delta_\pm, m_0,m_1\in \Z\}.\]
 This gives a triangular decomposition
 \begin{align}\label{tridec1}\wt\fg=\wt\fg^+\oplus \wt\fg^0\oplus \wt\fg^-\end{align}
for the  nullity $2$ toroidal EALA $\wt\fg$, where  $\wt\fg^\pm=\oplus_{\al\in \wt\Delta^\pm}\wt\fg_\al$ and
\[\wt\fg^0=\oplus_{\al\in \wt\Delta^0}\wt\fg_\al=\(\mathcal R\ot \dfh\)\oplus \mathcal K\oplus \mathcal S.\]

We say that a $\wt\fg^0$-module is a level $0$ weight module if $\rd_0, \rd_1$ act semi-simple and $\rk_0, \rk_1$ act trivially.
By using the decomposition \eqref{tridec1}, we define the following notion of highest weight module for $\wt\fg$.

\begin{dfnt}\label{defhw1} A $\wt\fg$-module $V$ is called a  highest weight module of type I
 if there exists a non-zero vector $v$ in $V$ such that
\begin{enumerate}
\item[(1)]\ $V$ is generated by $v$;
\item[(2)]\ $\wt\fg^+.v=0$;
\item[(3)]\  $\U(\wt\fg^0).v$ is an irreducible level $0$ weight $\wt\fg^0$-module.
\end{enumerate}
\end{dfnt}

Given an irreducible level $0$ weight $\wt\fg^0$-module $T$.
We extend $T$ to be a $(\wt\fg^+\oplus \wt\fg^0)$-module by letting $\wt\fg^+$ acts trivially on it.
Form the induced $\wt\fg$-module
\[ M(T)=\U(\wt\fg)\ot_{\U(\wt\fg^+\oplus \wt\fg^0)} T.\]
We denote by $V(T)$ the unique irreducible quotient of the $\wt\fg$-module $M(T)$.
If $V$ is an irreducible highest weight $\wt\fg$-module of type I, then
\[T=\{v\in V\mid \wt\fg^+.v=0\}\]
is an irreducible level $0$ weight $\wt\fg^0$-module and
$V$ is isomorphic to $V(T)$.

Now we are going to construct a class of irreducible level $0$ weight $\wt\fg^0$-modules.
We denote by
\[\dot{P}_+=\{\lambda\in \dfh^*\mid \lambda(\al_i^\vee)\in \N,\ i=1,\cdots,\ell\}\]
 the set of dominant integral weights of $\dg$.
Let $U$ be an irreducible finite dimensional $\mathfrak{sl}_2$-module, $\lambda\in \dot{P}_+\setminus\{0\}$ and $\bm{\gamma}=(\gamma_1,\gamma_2), \bm{\gamma}'=(\gamma_1',\gamma_2')\in \C^2$.
We define a $\wt\fg^0$-module $T_{U,\lambda,\bm{\gamma},\bm{\gamma}'}$ associated to the quadruple $(U,\lambda,\bm{\gamma},\bm{\gamma}')$ as follows.
The underlying space of  $T_{U,\lambda,\bm{\gamma},\bm{\gamma}'}$ is the tensor product space $\CR\ot U=\C[t_0^{\pm 1}, t_1^{\pm 1}]\ot U$  and the  actions are given by
\begin{align*}
t^{\bm{m}}\ot h. t^{\bm{n}}\ot u &=\lambda(h)\,t^{\bm{m}+\bm{n}}\ot u,\\
t^{\bm{m}}\rk_i. t^{\bm{n}}\ot u &=0,\\
 \rd_i.t^{\bm{n}}\ot u &=(n_i+\gamma'_i)\,t^{\bm{n}}\ot u,\\
\rd_{\bm{m}}.t^{\bm{n}}\ot u &=t^{\bm{m}+\bm{n}}\ot \(\(
                             \begin{array}{cc}
                              -m_0m_1 & m_0^2\\
                               -m_1^2 & m_0m_1 \\
                             \end{array}
                           \)
+\det {\bm{m}\choose \bm{\gamma}+\bm{n}}\).u,
\end{align*}
where $\bm{m},\bm{n}\in \Z^2$, $u\in U$, $h\in \dfh$ and $i=0,1$.
It is obvious that $T_{U,\lambda,\bm{\gamma},\bm{\gamma}'}$ is an irreducible level $0$ weight $\wt\fg^0$-module.
The following classification result will be proved in Section \ref{section4}.

\begin{thm}\label{mainhw1} Let $U$ be an irreducible  finite dimensional $\mathfrak{sl}_2$-module, $\dot\lambda\in \dot{P}_+\setminus\{0\}$ and
$\bm{\gamma},\bm{\gamma}'\in \C^2$.
Then the irreducible highest weight $\wt\fg$-module $V(T_{U,\lambda,\bm{\gamma},\bm{\gamma}'})\in \vartheta_{\mathrm{fin}}^\times$.
Conversely, if $V\in \vartheta_{\mathrm{fin}}^\times$ is an irreducible  highest weight $\wt\fg$-module of type I, then
$V$ has the form $V(T_{U,\lambda,\bm{\gamma},\bm{\gamma}'})$.
\end{thm}

\subsection{Integrable highest weight $\wt\fg$-modules of type II}
In this subsection we introduce another kind of highest weight module for the  Lie algebra $\wt\fg$,
and also state the sufficient and necessary condition for  such an irreducible highest  weight $\wt\fg$-module belongs to the category $\vartheta_{\mathrm{fin}}^\times$.

We denote by
\[\fg=\(\C[t_0,t_0^{-1}]\ot \dg\)\oplus \C\rk_0\oplus \C\rd_0,\]
which is an affine subalgebra of $\wt\fg$.
Let
 $\fh=\dfh\oplus \C\rk_0\oplus \C\rd_0$ be the usual Cartan subalgebra of $\fg$. Then the set
\[\Delta=\{\al+m_0\delta_0\mid \al\in \dot\Delta,\ m_0\in \Z\}\]
is the root system of $\fg$ relative to the Cartan subalgebra $\fh$.
And the set  $\Pi=\{\al_0=\delta_0-\theta,\,\al_i\mid i=1,\cdots,\ell\}$ is a simple root system of $\fg$, where
 $\theta$ indicates the highest root in $\dot{\Delta}$.
Denote by $\Delta_+$ and $\Delta_-$ respectively the corresponding positive and negative root systems of $\fg$.
Then we have another decomposition
\[\wt\Delta=\wt\Delta_+\cup \wt\Delta_0\cup \wt\Delta_-\]
of $\wt\Delta$, where
\[\wt\Delta_\pm=\{\al+m_1\delta_1\mid \al\in \Delta_\pm, m_1\in \Z\}\ \te
 {and}\
 \wt\Delta_0=\{m_1\delta_1\mid m_1\in \Z\}.\]
As in the previous subsection, this decomposition of $\wt\Delta$ induces a triangular decomposition
\begin{align}\label{tridec2}\wt\fg=\wt\fg_+\oplus \wt{\mathcal H}\oplus \wt\fg_-\end{align}
for the  nullity $2$ toroidal EALA $\wt\fg$, where  $\wt\fg_\pm=\oplus_{\al\in \wt\Delta_\pm}\wt\fg_\al$, and
\begin{align}\label{ch}
\wt{\mathcal H}=\oplus_{\al\in \wt\Delta_0}\wt\fg_\al=\(\C[t_1,t_1^{-1}]\ot \dfh\)\oplus \sum_{n\in \Z}\(\C t_1^n\rk_0\oplus \C t_1^n\rd_0\)\oplus \C\rk_1\oplus \C\rd_1.\end{align}

We say that an $\wt\CH$-module is a level $0$ weight module if $\rd_1$ acts semi-simple and $\rk_1$ acts trivially.
Similar to Definition \ref{defhw1}, by using the decomposition \eqref{tridec2} of $\wt\fg$, we have the following definition.

\begin{dfnt}\label{defhw2} A $\wt\fg$-module $V$ is called a highest weight module of type II if there exists a non-zero vector $v$ in $V$ such that
\begin{enumerate}
\item[(1)]\ $V$ is generated by $v$;
\item[(2)]\ $\wt\fg_+.v=0$;
\item[(3)]\  $\U(\wt\CH).v$ is an irreducible level $0$ weight $\wt\CH$-module.
\end{enumerate}
\end{dfnt}


Note that the algebra $\wt\fg$ is $\Z$-graded with respect to the action of $\rd_1$. Namely,
\begin{align}\label{zgraded}
\wt\fg=\oplus_{n\in \Z}\wt\fg_n,\end{align}
where $\wt\fg_n=\{x\in \wt\fg\mid [\rd_1,x]=nx\}$. For any graded subalgebra $\mathcal L$ of $\wt\fg$,
the corresponding homogenous subspaces will be denoted  by $\mathcal L_n$ for $n\in \Z$.

We denote by
\begin{align}\label{hatch}
\wh\CH=\(\C[t_1,t_1^{-1}]\ot \dfh\)\oplus \sum_{n\in \Z}\(\C t_1^n\rk_0\oplus \C t_1^n\rd_0\)\oplus \C\rk_1,\end{align}
which is a Heisenberg subalgebra of $\wt\fg$ and $\wt\CH=\wh\CH\oplus \C\rd_1$.
Let $\psi$ be a linear functional on  $\wh\CH$ with $\psi(\rk_1)=0$.
Using the functional $\psi$, we can define an $\wt\CH$-module structure on the Laurent polynomial ring $\C[t_1,t_1^{-1}]$ with  action given by
\[ h.t_1^m=\psi(h)t_1^{m+n},\quad \rd_1.t_1^m=m\,t_1^m,\]
for $m,n\in \Z$ and $h\in \wh\CH_n$.
The resulting $\wt\CH$-module  is a level $0$ weight $\wt\CH$-module and will be denoted by $L(\psi)$.
We write
$L_{\psi}=\U(\wt\CH).1$ for
the $\wt\CH$-submodule of $L(\psi)$ generated by $1$.
For any $b\in\C$, one can modify the action of $\rd_1$ on
$L_{\psi}$ by adding the scalar action $b\, \mathrm{Id}$. The resulting $\wt\CH$-module
will be denoted by $L_{\psi,b}$.

 For any non-negative integer $r$, set
 \begin{equation*}L_r=\begin{cases}
\C[t_1^r,t_1^{-r}],\quad &\te{if}\quad r>0;\\
\C1,\quad&\te{if}\quad r=0.\end{cases}\end{equation*}
 The following results are well-known (\cite{C}).
\begin{lemt}\label{lem:hmod} (i). Let $\psi$ be a linear functional on $\wh\CH$ with $\psi(\rk_1)=0$, and $b\in \C$.
Then the $\wt\CH$-module $L_{\psi,b}$ is irreducible if and only if as vector spaces, $L_{\psi}=L_r$ for some $r\in \N$.
Furthermore, if $L_{\psi}=L_r$ for some $r>0$, then the $\wt\CH$-module $L(\psi)$ is completely reducible
\[L(\psi)=\bigoplus_{i=0}^{r-1}\U(\wt\CH).t_1^i\] and each irreducible component $\U(\wt\CH).t_1^i=t_1^iL_r$ is isomorphic to $L_{\psi,i}$.

(ii).  Any irreducible level $0$ weight $\wt\CH$-module has the form $L_{\psi,b}$.
\end{lemt}

We denote by $\CE$ the set of linear functionals $\psi$ on $\wh\CH$ such that $\psi(\rk_1)=0$ and that the associated
$\wt\CH$-module $L_{\psi}$ is irreducible.
For any $\psi\in \CE$ and $b\in \C$, we extend the  $\wt\CH$-module $L_{\psi,b}$
to be a $(\wt\fg_+\oplus \wt\CH)$-module by letting $\wt\fg_+$ acts trivially on it.
Then we have the induced $\wt\fg$-module
\[ \widetilde{M}(\psi,b)=\U(\wt\fg)\ot_{\U(\wt\fg_+\oplus \wt\CH)} L_{\psi,b}.\]
Let $\widetilde{V}(\psi,b)$ stand for the unique irreducible quotient of the $\wt\fg$-module $\widetilde{M}(\psi,b)$.
Suppose now that $V$ is a highest weight $\wt\fg$-module of type II. Then, by Lemma \ref{lem:hmod} (ii), $V$ is a quotient of $\widetilde{M}(\psi,b)$ for some $\psi\in\CE$ and $b\in \C$. Furthermore,
if $V$ is also irreducible, then $V$ must be isomorphic to $\widetilde{V}(\psi,b)$.

Denote by
\[P_+=\{\lambda\in \fh^*\mid \lambda(\al_i^\vee)\in \N,\ i=0,1,\cdots,\ell\}\]
the set of dominant integral weights of $\fg$, where $\al_0^\vee=\rk_0-\theta^\vee$.
Recall from \cite{BZ} that a function $f:\Z\rightarrow \C$ is called
exp-polynomial if there exist finitely many $c_1,\cdots,c_p,b_1,\cdots,b_p\in \C^\times$ and $m_1,\cdots,m_p\in \N$ such that
\[f(n)=\sum_{i=1}^p c_i\, n^{m_i}\, b_i^{n}\]
for $n\in \Z$.

Let $k$ be a fixed positive integer, and let $(\bm{\lambda},\bm{a},\varphi)$ be a triple which satisfies the following conditions
\begin{enumerate}
\item[(c1)] $\bm{\lambda}=(\lambda_1,\cdots,\lambda_k)\in (P_+)^k$ with $\lambda_i\ne 0$ for all $i$;
\item[(c2)] $\bm{a}=(a_1,\cdots,a_k)\in (\C^\times)^k$ with $a_1,\cdots,a_k$  distinct;
\item[(c3)] $\varphi:\Z\rightarrow \C$ is an exp-polynomial function such that $\varphi(0)=0$.
\end{enumerate}
We now define a linear functional
$\psi_{\bm{\lambda},\bm{a},\varphi}:\wh\CH\rightarrow \C$ associated to the triple $(\bm{\lambda},\bm{a},\varphi)$ as follows
\begin{align*}
\psi_{\bm{\lambda},\bm{a},\varphi}(t_1^n\ot h)=\sum_{i=1}^k\lambda_i(h)a_i^n,\quad
\psi_{\bm{\lambda},\bm{a},\varphi}(t_1^n\rk_0)=\sum_{i=1}^k\lambda_i(\rk_0)a_i^n,\\
\psi_{\bm{\lambda},\bm{a},\varphi}(t_1^m\rd_0)=\frac{\varphi(m)}{m^2},\quad
\psi_{\bm{\lambda},\bm{a},\varphi}(\rk_1)=0,\quad \psi_{\bm{\lambda},\bm{a},\varphi}(\rd_0)=\sum_{i=1}^k\lambda_i(\rd_0),
\end{align*}
where $n\in \Z$, $h\in \dfh$ and $m\in \Z\setminus\{0\}$.
It is easy to see that the linear functional $\psi_{\bm{\lambda},\bm{a},\varphi}\in \CE$.
The following result will be proved in Section \ref{section5}.

%

\begin{thm}\label{mainhw2} Let $(\bm{\lambda},\bm{a},\varphi)$ be a triple which satisfies the conditions (c1)-(c3) and let $b\in \C$.
Then the irreducible highest weight $\wt\fg$-module $\widetilde{V}(\psi_{\bm{\lambda},\bm{a},\varphi},b)\in \vartheta_{\mathrm{fin}}^\times$. Conversely, if $V\in \vartheta_{\mathrm{fin}}^\times$ is an irreducible highest weight $\wt\fg$-module of type II, then it  has the form $\widetilde{V}(\psi_{\bm{\lambda},\bm{a},\varphi},b)$.
\end{thm}

\subsection{The main result}

In this subsection we state and prove the main result of our paper by using Theorem \ref{mainhw1} and Theorem \ref{mainhw2}.
We begin with the following result whose proof is standard, see \cite[Proposition 2.4]{E2} or \cite[Theorem 2.1]{JM} for example.
\begin{prpt}\label{redhw} Let $V\in \vartheta_{\mathrm{fin}}$ be  irreducible with central charge $(c_0,c_1)$.

(i) If $c_0=c_1=0$,   then there exists a non-zero weight vector $v\in V$ such that
$\wt\fg^+.v=0.$

(ii). If $c_0>0$ and $c_1=0$, then there exists a non-zero weight vector $v\in V$ such that
$\wt\fg_+.v=0.$
\end{prpt}

Let $A=(a_{ij})_{i,j=0,1}$ be a  matrix in $\mathrm{GL}_2(\Z)$.
As in \cite[Section 4]{E2}, there is an automorphism $\chi_A$ of $\wt\fg_c$ such that
\begin{align*}
\chi_A(t^{\bm{m}}\otimes x)=t^{\bm{m}A^t}\otimes x,\quad
\chi_A(t^{\bm{m}}\rk_j)=\sum_{i=0,1}a_{ij}t^{\bm{m}A^t}\rk_i,\end{align*}
  where $x\in\dot{\fg}$, $\bm{m}\in \Z^2$, $j=0,1$ and $A^t$ stands for the transpose of $A$.
Write $B=(b_{ij})_{i,j=0,1}$ for the inverse matrix of $A$. Then the map $\chi_A$ can be extended to an automorphism of $\wt\fg$, still call $\chi_A$,
such that
\begin{align*}
\chi_A(\rd_j)=\sum_{i=0,1}b_{ji}\rd_i,\quad
\chi_A(\rd_{\bm{m}})=(\det B)\,
\rd_{\bm{m}A^t},
\end{align*}
where $j=0,1$ and $\bm{m}\in \Z^2$.

For any $\wt\fg$-module $W$,
by  the automorphism $\chi_A$ twisting, one obtains another module structure on $W$. The resulting
$\wt\fg$-module will be
denoted by $W_A$. We say that
two $\wt\fg$-modules $W$ and $W'$ are isomorphism
 after a change of coordinates if $W'$ is isomorphic to $W_A$ for some $A\in \mathrm{GL}_2(\Z)$.

Now we are in a position to  prove the main result of this paper.

\begin{thm}\label{main} Let $V$ be an irreducible integrable $\wt\fg$-module with finite dimensional weight spaces.

(i). If $V$ has zero central charge and  $\wt\fg_c$ acts non-trivially on it, then $V$ is isomorphic to the type I highest
 weight $\wt\fg$-module $V(T_{U,\lambda,\bm{\gamma},\bm{\gamma}'})$,
 where $U$ is an irreducible finite dimensional $\mathfrak{sl}_2$-module, $\lambda\in \dot{P}_+\setminus \{0\}$ and
 $\bm{\gamma},\bm{\gamma}'\in \C^2$.

(ii). If $V$ has non-zero central charge, then after a change of coordinates,   $V$ is isomorphic to the type II
 highest weight $\wt\fg$-module $\widetilde{V}(\psi_{\bm{\lambda},\bm{a},\varphi},b)$,
 where  $(\bm{\lambda},\bm{a},\varphi)$  is a triple which satisfies the conditions (c1)-(c3), and $b\in \C$.
\end{thm}
\begin{proof} Assume first that $V\in \vartheta_{\mathrm{fin}}^\times$ is irreducible and has zero central charge.
By applying  Proposition \ref{redhw} (i), we know that
\[T=\{v\in V\mid \wt\fg^+.v=0\}\]
is a  non-zero $\wt\fg^0$-submodule of $V$.
The irreducibility of $V$ implies that the  $\wt\fg^0$-module $T$ is also irreducible.
Therefore, $V$ is isomorphic to the type I highest weight $\wt\fg$-module $V(T)$. Then
the assertion (i) follows from Theorem \ref{mainhw1}.

Secondly, let $V\in \vartheta_{\mathrm{fin}}$ be irreducible and has non-zero central charge $(c_0,c_1)$.
Note that one can choose a suitable matrix $A\in \mathrm{GL}_2(\Z)$ such that the $\wt\fg$-module $V_A$
has central charge $(c_0,c_1)A=(r,0)$  for some $r>0$.
Thus, after a change of coordinates, we may assume that $c_0>0$ and $c_1=0$.
Then, by Proposition \ref{redhw} (ii), one finds a non-zero $\wt\CH$-submodule
\[\{v\in V\mid \wt\fg_+.v=0\}\]
of $V$,
which is obviously irreducible. This gives that $V$ is a type II highest weight module for $\wt\fg$.
Therefore the assertion (ii) follows from  Theorem \ref{mainhw2}.
\end{proof}

\section{Proof of Theorem \ref{mainhw1}}\label{section4}
This section is devoted to the proof of Theorem \ref{mainhw1}.
\subsection{Integrable representations for extended loop algebras}
In this subsection we collect some basic results on integrable representations of the extended loop algebra of $\mathfrak{sl}_2$
for later use, see \cite{CP2} for details.

The extended loop algebra  of $\mathfrak{sl}_2$ is the Lie algebra
\[L^e(\mathfrak{sl}_2)=(\C[t,t^{-1}]\ot \mathfrak{sl}_2)\oplus \C\rd\]
with commutator given by
\[[t^m\ot x, t^n\ot y]=t^{m+n}\ot [x,y],\quad [\rd, t^m\ot x]=m\, t^m\ot x\]
for $m,n\in \Z$, $x,y\in \mathfrak{sl}_2$. Let $L(\mathfrak{sl}_2)=\C[t,t^{-1}]\ot \mathfrak{sl}_2$ be the loop subalgebra of $L^e(\mathfrak{sl}_2)$, and let $\{\re,\rh,\rf\}$ be a standard basis of $\mathfrak{sl}_2$ such that
\[[\re,\rf]=\rh,\quad [\rh,\re]=2\re,\quad [\rh,\rf]=-2\rf.\]
For $m\in \Z$, we define an elements $\Lambda_m\in \U(L^e(\mathfrak{sl}_2))$ as follows
\begin{align*}
\sum_{m=0}^\infty \Lambda_{\pm m}z^m =\mathrm{exp}\(-\sum_{k=1}^\infty \frac{t^{\pm k}\ot \rh}{k}z^k\).
\end{align*}

A module $V$ of $L^e(\mathfrak{sl}_2)$ (resp. $L(\mathfrak{sl}_2)$) is called integrable if $V$ admits a weight space decomposition
 relative to $\C\rh\oplus \C \rd$
  (resp. $\C\rh$), and  $t^m\ot \re$, $t^m\ot \rf$ act locally nilpotent on $V$ for
$m\in \Z$.
Let $V$ be an integrable module for $L^e(\mathfrak{sl}_2)$ or $L(\mathfrak{sl}_2)$. For $n\in \Z$, we set
\[V_n=\{v\in V\mid \rh.v=n v\},\quad
V_n^+=\{v\in V_n\mid t^m\ot \re.v=0,\  m\in \Z\}.\]
Note that $V_n^+\ne 0$ only if $n\in \N$.

\begin{lemt}\label{eloop1} Let $V$ be an integrable $L^e(\mathfrak{sl}_2)$-module, and $0\ne v\in V_n^+$ for some $n\in \N$.
Then  \begin{enumerate}
\item[(i)]\ $\Lambda_{m}.v=0$\ for\ $|m|>n$;
\item[(ii)]\ $\Lambda_n\Lambda_{-n}.v=v$;
\item[(iii)]\ if $t^{m}\ot \rh.v=0$ for all $m>0$, then $n=0$.
\end{enumerate}
\end{lemt}
\begin{proof} The assertion (i) follows from   \cite[Proposition 1.1 (i), (iv)]{CP2}, the assertion (ii) follows from \cite[Proposition 1.1 (v)]{CP2}, and the assertion (iii) can be deduced from (ii).
\end{proof}

\begin{lemt}\label{eloop2} Let $V$ be an integrable module for $L(\mathfrak{sl}_2)$  with finite dimensional weight spaces.
 Assume that $V$ is generated by a vector  $v\in V_n^+$ for some $n\in \N$ and $\dim V_n=1$.
Then there exist some non-negative integers $n_1,\cdots,n_k$ and distinct non-zero complex numbers  $p_1,\cdots,p_k$
such that
\[t^m\ot \rh.v=(\sum_{i=1}^k n_i p_i^m)v,\quad m\in \Z.\]
\end{lemt}
\begin{proof} Note that the integrability of $V$ forces that  $V=\oplus_{|m|\le n}V_m$.
Thus, $V$ is a finite dimensional $L(\mathfrak{sl}_2)$-module. Then the lemma is implied by  \cite[Proposition 2.1 (iii)]{CP2}.
\end{proof}
\subsection{Vanishing of central operators}
The main goal of this subsection is to  prove that the subalgebra $\mathcal K$ acts trivially on any irreducible integrable highest weight
$\wt\fg$-module of type I.
Throughout this subsection, let $T$ be an irreducible weight $\wt\fg^0$-module such that
the associated irreducible highest weight $\wt\fg$-module $V(T)$ is integrable.

\begin{lemt}\label{existlambda}  There exists a dominant integral weight $\lambda\in \dot{P}_+$ such that
\begin{align} h.v=\lambda(h)v,\end{align}
 for $h\in \dfh$, $v\in T$.
\end{lemt}
\begin{proof} We remark that $\dfh$ lies in the center of $\wt\fg^0$.
The irreducibility of $T$ yields a linear functional $\lambda$ on
$\dfh$ such that $h.v=\lambda(h)v$  for $h\in \dfh$, $v\in T$.
Furthermore, the integrability of $V$ forces that $\lambda\in \dot{P}_+$, as desired.
\end{proof}

For any $\bm{m}\in \Z^2$ and $k\in \Z$, set
\[\rk_{\bm{m},k}=m_0' t^{k\bm{m}}\rk_0+ m_1' t^{k\bm{m}}\rk_1\quad \te{and}\quad \rk_{\bm{m}}=\rk_{\bm{m},1},\] where
\begin{equation*} \bm{m}'=(m_0',m_1')=\begin{cases} (\frac{1}{m_1},0),\ &\te{if}\ m_1\ne 0;\\
(0,-\frac{1}{m_0}),\ &\te{if}\ m_0\ne 0, m_1=0;\\
(0,0),\ &\te{if}\ m_0=0=m_1.\end{cases}
\end{equation*}
We remark that the set
\[\{\rk_0,\rk_1,\rk_{\bm{m}}\mid\bm{0}\ne \bm{m}\in \Z^2\}\]
form a basis of $\mathcal K$. Moreover,
the following two lemmas can be easily checked.
\begin{lemt} For any $\bm{m},\bm{n}\in \Z^2$, the following holds in $V(T)$
\begin{align}\label{liedk}
[\rd_{\bm{m}}, \rk_{\bm{n}}]=\detmn \rk_{\bm{m}+\bm{n}}.
\end{align}
\end{lemt}

For every $\al\in \dot\Delta_+$,  $\bm{m}\in \Z^2\setminus \{\bm{0}\}$ and $n\in \Z$, we set
\begin{align*}
\wt{\mathfrak{sl}_2}(\al,\bm{m},n)=
\te{Span}_\C\{&\ t^{k\bm{m}+n\bm{m}'}\ot x_{\al}^+,\ t^{k\bm{m}-n\bm{m}'}\ot x_{\al}^-,\
t^{k\bm{m}}\ot \al^\vee+\frac{2n}{\<\al,\al\>}\rk_{\bm{m},k}, \\
&\ m_0\rk_0+m_1\rk_1,\ m_0''\rd_0+m_1''\rd_1\mid k\in \Z\},
\end{align*}
 where  $x_{\al}^\pm$ are root vectors in $\wt\fg_{\pm \al}$ such that
$\{x_{\al}^+,\al^\vee,x_{\al}^-\}$ is a $\mathfrak{sl}_2$-triple and
\begin{equation*} (m_0'',m_1'')=\begin{cases} (0,\frac{1}{m_1}),\ \te{if}\ m_1\ne 0;\\
(\frac{1}{m_0},0),\ \te{if}\ m_1=0, m_0\ne 0.\end{cases}
\end{equation*}
\begin{lemt}\label{A11}For every $\al\in \dot\Delta_+$,  $\bm{m}\in \Z^2\setminus \{\bm{0}\}$ and $n\in \Z$, $\wt{\mathfrak{sl}_2}(\al,\bm{m},n)$ is a subalgebra of $\wt\fg$ and
is  isomorphic to the affine Kac-Moody Lie algebra of type $A_1^{(1)}$.
\end{lemt}

Let $\lambda\in \dot{P}_+$ be as in Lemma \ref{existlambda} and set
\[N=\min\,\{\lambda(\al^\vee)+1\mid \al\in \dot\Delta_+\}.\]
\begin{lemt}\label{keylemma}  For all  $\bm{m}\in \Z^2$ and $v\in T$, one has that
\begin{align} \rk_{\bm{m}}^{N}.v=0.\end{align}

\end{lemt}
\begin{proof} Let $\al\in \dot\Delta_+$, $\bm{m}\in \Z^2\setminus\{\bm{0}\}$ and $n\in \Z$.
Since $\rk_0$, $\rk_1$ act trivially on $V(T)$, it follows from Lemma \ref{A11} that 
there is  an integrable $L^e(\mathfrak{sl}_2)$-module structure on $V(T)$ by using the subalgebra $\wt{\mathfrak{sl}}_2(\al,\bm{m},n)$.
The resulting $L^e(\mathfrak{sl}_2)$-module  is denoted by $V(T,\al,\bm{m},n)$.
We remark that $T\subset V(T,\al,\bm{m},n)_{\lambda(\al^\vee)}^+$.
Thus, by applying Lemma \ref{eloop1} (i), one obtains
\begin{align} \Lambda_b^{\al,\bm{m},n}.v=0,
\end{align}
for $v\in T$ and $b>\lambda(\al^\vee)$,
where the operators $\Lambda_b^{\al,\bm{m},n}\in \mathrm{End}(V(T))$ are defined as follows
\begin{align}\label{keylemma1}
\sum_{b\in \N}\Lambda_b^{\al,\bm{m},n}z^b=\mathrm{exp}\(-\sum_{k=1}^\infty \frac{t^{k\bm{m}}\ot \al^\vee+\frac{2n}{\<\al,\al\>}\rk_{\bm{m},k}}{k}z^k\).
\end{align}

We now introduce more operators $\Lambda_{1,b}^{\al,\bm{m}}, \Lambda_{2,b}^{\al,\bm{m},n}, \Lambda_{2,b}^{\bm{m}}(s)$ on $V(T)$
for $b,s\in \N$  as follows
\begin{align*}
\sum_{b\in \N}\Lambda_{1,b}^{\al,\bm{m}}z^b=\mathrm{exp}\(-\sum_{k=1}^\infty \frac{t^{k\bm{m}}\ot \al^\vee}{k}z^k\),\\
\sum_{b\in \N} \Lambda_{2,b}^{\al,\bm{m},n}z^b=\mathrm{exp}\(-\sum_{k=1}^\infty \frac{2n\rk_{\bm{m},k}}{\<\al,\al\>k}z^k\),
\end{align*}
and
\begin{equation*}\Lambda_{2,b}^{\bm{m}}(s)=\begin{cases}\delta_{b,0},\ &\te{if}\ s=0;\\
\sum_{1\le k_1,\cdots,k_s\le b;k_1+\cdots+k_s=b}\frac{(-1)^s}{s!}\frac{\rk_{\bm{m},k_1}}{k_1}\cdots \frac{\rk_{\bm{m},k_s}}{k_s},
\ &\te{if}\ 1\le s\le b;\\
0,\ &\te{if}\ s>b.\end{cases}\end{equation*}
Then it is easy to see that
\begin{align*}
\Lambda_b^{\al,\bm{m},n}&=\sum_{b_1,b_2\ge 0;b_1+b_2=b}\Lambda_{1,b_1}^{\al,\bm{m}}\cdot\Lambda_{2,b_2}^{\al,\bm{m},n},\\
\Lambda_{2,b}^{\al,\bm{m},n}&=\sum_{s=0}^b \(\frac{2n}{\<\al,\al\>}\)^s \Lambda_{2,b}^{\bm{m}}(s),
\end{align*}
for $b\in \N$.
Using the above two equations, one gets
\begin{align}\label{keylemma2}
\Lambda_b^{\al,\bm{m},n}=\sum_{s=0}^b\(\frac{2n}{\<\al,\al\>}\)^s\,\Lambda_b^{\al,\bm{m}}(s),
\end{align}
for $m\in \N$, $n\in \Z$,
where
\begin{align*} \Lambda_b^{\al,\bm{m}}(s)=\sum_{b_1,b_2\ge 0;b_1+b_2=b}\Lambda_{1,b_1}^{\al,\bm{m}}\cdot  \Lambda_{2,b_2}^{\bm{m}}(s).
\end{align*}

 Now, for any given $\al\in \dot{\Delta}_+$, $\bm{m}\in \Z^2\setminus \{\bm{0}\}$, $b>\lambda(\al^\vee)$ and $v\in T$, by applying \eqref{keylemma1} and \eqref{keylemma2}, we have the following  system of $b+1$ equations
\begin{align} \sum_{s=0}^b \(\frac{2n}{\<\al,\al\>}\)^s \Lambda_b^{\al,\bm{m}}(s).v=0,\quad n=0,\cdots,b.
\end{align}
 By solving this system of equations, we get that
\begin{align*}\Lambda_b^{\al,\bm{m}}(s).v=0,
\end{align*}
for $0\le s\le b$.
This together with the fact that
\[\Lambda_b^{\al,\bm{m}}(b)= \Lambda_{2,b}^{\bm{m}}(b)=\frac{(-1)^b}{b!}\,\rk_{\bm{m}}^b\] gives
\begin{align} \label{mk0}
\rk_{\bm{m}}^b.v=0,
\end{align}
 as required.

\end{proof}

\begin{prpt}\label{vanishing} One has that \[\mathcal K. V(T)=0.\]
\end{prpt}
\begin{proof} Notice first that $\mathcal K$ commutes with $\wt\fg^-$ and $V(T)=\U(\wt\fg^-)T$.
Therefore it suffices to prove that $\mathcal K.T=0$. Given an $\bm{n}\in \Z^2\setminus\{\bm{0}\}$.
Using  \eqref{liedk}, for any $\bm{m}_1\in\Z^2$, the following holds in $\mathrm{End}(V(T))$
\[\rd_{\bm{m}_1}\rk_{\bm{n}}^N=N\det{\bm{m}_1\choose \bm{n}}\rk_{\bm{m}_1+\bm{n}}\rk_{\bm{n}}^{N-1}+\rk_{\bm{n}}^N\rd_{\bm{m}_1}.\]
This together with Lemma \ref{keylemma} gives that, if $\epsilon_1\bm{m}_1\notin\mathbb Q \bm{n}$, $\epsilon_1\in \{0,1\}$, then
\begin{align}\label{induc1}
\rk_{\bm{m}_1+\bm{n}}\rk_{\bm{n}}^{N-1}.T=0. \end{align}
Again by \eqref{liedk}, for any $\bm{m}_2\in \Z^2$,  one has that
\[\rd_{\bm{m}_2}\rk_{m_1+n}\rk_{\bm{n}}^{N-1}=(N-1)\det{\bm{m}_2\choose \bm{n}}\rk_{\bm{m}_1+\bm{n}}\rk_{\bm{m}_2+\bm{n}}\rk_{\bm{n}}^{N-2}+\det{\bm{m}_2\choose \bm{m}_1+\bm{n}}\rk_{\bm{m}_1+\bm{m}_2+\bm{n}}\rk_{\bm{n}}^{N-1}\]
 as operators on $V(T)$.
Thus, by \eqref{induc1}, if $\epsilon_1\bm{m}_1+\epsilon_2\bm{m}_2\not\in \mathbb Q \bm{n}$, $\epsilon_1,\epsilon_2\in \{0,1\}$, then
\[\rk_{\bm{m}_1+\bm{n}}\rk_{\bm{m}_2+\bm{n}}\rk_{\bm{n}}^{N-2}.T=0.\]
By repeating the above argument, one easily gets that
\begin{align}\label{induc2}\rk_{\bm{m}_1+\bm{n}}\rk_{\bm{m}_2+\bm{n}}\cdots \rk_{\bm{m}_N+\bm{n}}.T=0,\end{align}
for all $\bm{m}_1,\cdots,\bm{m}_N\in \Z^2$ with $\sum_{i=1}^N\epsilon_i\bm{m}_i\notin \mathbb Q\bm{n}$, $\epsilon_i\in \{0,1\}$.

Let $\bm{n}_1,\cdots,\bm{n}_N\in \Z^2\setminus\{0\}$. Note that  one can choose an $\bm{n}\in \Z^2$ such that $\sum_{i=1}^N\epsilon_i\bm{n}_i\notin \mathbb Q\bm{n}$, $\epsilon_i\in \{0,1\}$. By taking $\bm{m}_i=\bm{n}_i-\bm{n}$ in \eqref{induc2}, we have that
\[\rk_{\bm{n}_1} \cdots \rk_{\bm{n}_N}.T=0.\]
This implies that the $\wt\fg^0$-submodule
\[\{v\in T\mid \mathcal K.v=0\}
\]
of $T$ is non-zero and hence coincides with $T$, as required.
\end{proof}

\begin{cort}\label{vanishinglambda} Let $\lambda$ be as in Lemma \ref{existlambda}. If $\lambda=0$, then
 $\wt\fg_c.V(T)=0$.
\end{cort}
\begin{proof}
Let $\al\in \dot{\Delta}_+$ and $\bm{m}\in \Z^2$.  By applying Proposition \ref{vanishing}, we know that
$\{t^{\bm{m}}\ot x_\al^+, t^{-\bm{m}}\ot x_{\al}^-, \al^\vee\}$ forms a $\mathfrak{sl}_2$-triple in $\mathrm{End}(V(T))$.
Now, for any $v\in T$, we  have that $t^{\bm{m}}\ot x_\al^+.v=0=\al^\vee.v$ and $t^{-\bm{m}}\ot x_\al^-$ acts nilpotently on $v$.
This gives that $t^{-\bm{m}}\ot x_{\al}^-.v=0$ and hence 
\[t^{-\bm{m}}\ot\al^\vee.v=[x_\al, t^{-\bm{m}}\ot x_{\al}^-].v=0.\]
Note that $\wt\fg_c$ is spanned by the space $\mathcal K$ and the elements $t^{\bm{m}}\ot x_{\al}^\pm, t^{\bm{m}}\ot \al^\vee$, $\al\in \dot\Delta_+, \bm{m}\in \Z^2$.
Thus, we have obtained that $\wt\fg_c.T=0$.
This indeed implies the lemma as
 $\wt\fg_c$ is an ideal of $\wt\fg$ and $V(T)=\U(\wt\fg)\,T$.

\end{proof}
\subsection{Irreducible uniformly bounded weight $\CG$-modules}
Let $\CG$ be the quotient algebra of $\wt\fg^0$ obtained by modulo the ideal $\mathcal K$.
Note that the Lie algebra $\CG$ is isomorphic to $\(\CR\ot \dfh\)\rtimes \wt\CS$,  the semi-product  of  $2$-loop algebra of $\dfh$ and $\wt\CS$.
In the following we will often identify $\CG$ with this semi-product Lie algebra. Then
the Lie brackets on $\CG$ are given by \eqref{wtcs} and
\[[h(\bm{m}),h'(\bm{n})]=0,\ [\rd(\bm{m}), h(\bm{n})]=\detmn h(\bm{m}+\bm{n}),\ [\rd_i,h(\bm{n})]=n_i\, h(\bm{n}),\]
where $\bm{m},\bm{n}\in \Z^2$, $h,h'\in \dfh$ and $h(\bm{m})=t^{\bm{m}}\ot h$.
We say that a $\CG$-module is a weight module if it admits a weight space decomposition
relative to $\C\rd_0\oplus \C\rd_1$,
and that a weight $\CG$-module is
uniformly bounded if there exists a positive integer $N$ such that the dimension of any  weight space of $V$ is
 less than $N$. In this subsection we will give a classification of the irreducible uniformly bounded weight $\CG$-modules.

We first construct a class of irreducible uniformly bounded weight $\CG$-modules.
Let $U$ be an irreducible finite dimensional $\mathfrak{sl}_2$-module,
$\lambda, \lambda'\in \dfh^*$ and $\bm{\gamma},\bm{\gamma}'\in \C^2$.
We define a $\CG$-module structure on the space $\CR\ot U$ with the actions given by
\begin{align*}
h(\bm{m}).t^{\bm{n}}\ot u&=(\lambda(h)+\delta_{\bm{m},\bm{0}}\lambda'(h))t^{\bm{m}+\bm{n}}\ot u,\\
\rd_i.t^{\bm{n}}\ot u&=(n_i+\gamma'_i) t^{\bm{n}}\ot u,\\
\rd(\bm{m}).t^{\bm{n}}\ot u&=t^{\bm{m}+\bm{n}}\ot \(\(
                             \begin{array}{cc}
                              -m_0m_1 & m_0^2\\
                               -m_1^2 & m_0m_1 \\
                             \end{array}
                           \)
+\det {\bm{m}\choose \bm{\gamma}+\bm{n}}\).u,
\end{align*}
where $\bm{m},\bm{n}\in \Z^2$, $u\in U$, $h\in \dfh$ and $i=0,1$.
The resulting $\CG$-module will be denoted by $T_{U,\lambda,\lambda',\bm{\gamma},\bm{\gamma}'}$.
It is obvious that, if $\lambda\ne 0$, then the $\CG$-module $T_{U,\lambda,\lambda',\bm{\gamma},\bm{\gamma}'}$ is
an irreducible uniformly bounded weight module and satisfies the condition
\begin{align}\label{assump} \te{the action of }
\CR'\ot \dfh\ \te{is non-trivial},\end{align}
where $\mathcal R'=\oplus_{\bm{0}\ne \bm{m}\in \Z^2} \C t^{\bm{m}}$.
Conversely, we have the following result.
\begin{prpt}\label{ubm} Let $T$ be an irreducible uniformly bounded weight $\CG$-module satisfying the condition \eqref{assump}.  Then $T$ is isomorphic to $T_{U,\lambda,\lambda',\bm{\gamma},\bm{\gamma}'}$, where $U$ is an irreducible finite dimensional $\mathfrak{sl}_2$-module,
$\lambda\in \dfh^*\setminus\{0\}$, $\lambda'\in \dfh^*$ and $\bm{\gamma},\bm{\gamma}'\in \C^2$.
\end{prpt}

The rest part of this subsection is devoted to a proof of Proposition \ref{ubm}.
 When  $\dim \dfh=1$, this proposition was proved in \cite{GL} and our proof  is similar to that in \cite{GL}.
 From now on, let $T$ be as in Proposition \ref{ubm}.
We start with the following three lemmas, whose proof are respectively similar to that in \cite[Lemma 3.3, Proposition 3.4, Lemma 3.6]{GL} and are omitted.
\begin{lemt}\label{unprove1} Let $x\in \U(\CR\ot \dfh)$. If $x.v=0$ for some $0\ne v\in T$, then $x$ acts locally nilpotent
on $V$.
\end{lemt}
\begin{lemt}\label{unprove2} Let $h\in \dfh$. Then either all $h(\bm{m})$, $\bm{m}\ne \bm{0}$ act injectively on $T$ or all
$h(\bm{m})$, $\bm{m}\ne \bm{0}$ act locally nilpotent on $T$.
\end{lemt}
\begin{lemt}\label{unprove3} Let $h\in \dfh$. If for any $\bm{m}\ne \bm{0}$,
$h(\bm{m})$  acts locally nilpotent on $T$, then $h(\bm{m}).T=0$ for all $\bm{m}\ne \bm{0}$.
\end{lemt}

We denote by
\[J=\{i=1,\cdots,\ell\mid \al_i^\vee(\bm{m})\ \te{acts injectively on}\ T,\ \forall\ \bm{m}\ne \bm{0}\}.\]
It follows from Lemma \ref{unprove2}, Lemma \ref{unprove3} and the condition \eqref{assump} that the set $J$ is non-empty.
We remark that there is a $\bm{\gamma}'=(\gamma_1',\gamma_2')\in \C^2$ such that
\begin{align}\label{decT}T=\oplus_{\bm{m}\in \Z^2}T(\bm{m}),\end{align}
where $T(\bm{m})=\{v\in T\mid \rd_i.v=(\gamma_i'+m_i)v,\ i=0,1\}$.
Moreover, there is a positive integer $k$ such that $\dim T(\bm{m})=k$ for all $\bm{m}\in \Z^2$.

\begin{lemt}\label{unprove4} There exist $\CR\ot \dfh$-submodules $W_0,W_1,\cdots,W_k$ of $T$ such that
\[W_0=\{0\}\subset W_1\subset W_2\subset \cdots \subset W_k=T,\]
and  $\dim(W_s(\bm{m}))=s$ for all $\bm{m}\in \Z^2$, where $W_s(\bm{m})=W_s\cap T(\bm{m})$.
\end{lemt}
\begin{proof} It is known that any irreducible $\Z^2$-graded $\CR\ot \dfh$-module
can be realized as a subspace of $\CR$ (\cite[Lemma 3.3]{E1}).
This gives that any irreducible $\CR\ot \dfh$-submodule of $T$ is isomorphic to $\CR$ (as vector spaces),
as required.
\end{proof}

For any $\bm{m}\in \Z^2$, we fix a basis $\{v_1(\bm{m}),\cdots,v_k(\bm{m})\}$ of $T(\bm{m})$ so that \[W_s(\bm{m})=\te{Span}_\C\{v_1(\bm{m}),\cdots,v_s(\bm{m})\}\] for $s=1,\cdots,k$. Note that $\dfh$ is the center of $\CG$.
Thus there is a $\lambda\in \dfh^*$ such that $h.v=\lambda(h)v$ for $h\in \dfh$ and $v\in T$.
Moreover, as $\dfh$ is not contained in the derived subalgebra of $\CG$, we can modify the action of $\dfh$ if necessary.
Assume first that $\lambda(\al_j^\vee)\ne 0$ for all $j\in J$.
For any $i,j\in J$ and $\bm{m},\bm{n}\in \Z^2$, we define
$\lambda_{\bm{m},\bm{n}}^{i,j}$ to be the unique non-zero complex number such that the action of
\[T_{\bm{m},\bm{n}}^{i,j}=\al_i^\vee(\bm{m})\al_i^\vee(\bm{n})-\lambda_{\bm{m},\bm{n}}^{i,j}\al_j^\vee(\bm{m}+\bm{n})\in \U(\CR\ot \dfh)\]
on $W_1(\bm{0})$ is trivial.

\begin{lemt}\label{unprove5} For any $i,j\in J$, $s=1,\cdots,k$ and $\bm{m},\bm{n},\bm{r}\in \Z^2$, one has that
\[T_{\bm{m},\bm{n}}^{i,j}.W_s(\bm{r})\subset W_{s-1}(\bm{m}+\bm{n}+\bm{r}).\]
\end{lemt}
\begin{proof}
We define a $k\times k$-matrix $B_{\bm{m},\bm{n}}^{i,j}$ by letting
\[T_{\bm{m},\bm{n}}^{i,j}(v_{1}(\bm{r}),\cdots,v_{k}(\bm{r}))=(v_{1}(\bm{m}+\bm{n}+\bm{r}),\cdots,v_{k}(\bm{m}+\bm{n}+\bm{r}))B_{\bm{m},\bm{n}}^{i,j}.\]
It follows from Lemma \ref{unprove4}  that the matrix $B_{\bm{m},\bm{n}}^{i,j}$ is upper triangular.
On the other hand, by applying Lemma \ref{unprove1} and the fact that $T^{i,j}_{\bm{m},\bm{n}}.v_1(\bm{0})=0$, one gets that $T^{i,j}_{\bm{m},\bm{n}}$ acts  nilpotently on $T(\bm{r})$.
This gives that the matrix $B_{\bm{m},\bm{n}}^{i,j}$ is nilpotent and hence is strictly upper triangular, as required.
\end{proof}

We now fix the action of $\dfh$ on $T$ so that $\lambda(\al_j^\vee)=0$ if $j\notin J$, and
$\lambda_{(1,0),(1,0)}^{j,j}=\lambda_{(1,0),(-1,0)}^{j,j}$ if $j\in J$.

\begin{lemt}\label{compatible} For any $i,j=1,\cdots,\ell$ and $\bm{m},\bm{n}\in \Z^2$, one has that
\[\(\lambda(\al_j^\vee)\al_i^\vee(\bm{m})\al_i^\vee(\bm{n})-(\lambda(\al_i^\vee))^2\al_j^\vee(\bm{m}+\bm{n})\).T=0.\]
\end{lemt}
\begin{proof}  Let $\bm{m},\bm{n},\bm{r}\in \Z^2$ and $i,j\in J$. Note that in $\U(\CG)$ one has
\begin{align*} [\rd(\bm{r}), T_{\bm{m},\bm{n}}^{i,j}]=\det{\bm{r}\choose \bm{m}}T_{\bm{m}+\bm{r},\bm{n}}^{i,j}
+\det{\bm{r}\choose \bm{n}}T_{\bm{m},\bm{n}+\bm{r}}^{i,j}+\lambda_{\bm{m},\bm{n},\bm{r}}^{i,j}\al_j^\vee(\bm{m}+\bm{n}+\bm{r}),
\end{align*}
where
\[\lambda_{\bm{m},\bm{n},\bm{r}}^{i,j}=\det{\bm{r}\choose \bm{m}}\lambda_{\bm{m}+\bm{r},\bm{n}}^{i,j}+\det{\bm{r}\choose \bm{n}}\lambda_{\bm{m},\bm{n}+\bm{r}}^{i,j}-\det{\bm{r}\choose \bm{m}+\bm{n}}\lambda_{\bm{m},\bm{n}}^{i,j}.\]
Using this,  one easily checks that there exist $\eta_1\in \mathrm{Hom}(T(\bm{0}),T((k-1)\bm{m}+(k-1)\bm{n}+k\bm{r}))$ and
$\eta_2,\eta_3\in \mathrm{Hom}(T(\bm{0}),T((k-1)\bm{m}+(k-1)\bm{n})+(k-1)\bm{r})$ such that,
as operators from  $T(\bm{0})$ to $T(k(\bm{m}+\bm{n}+\bm{r}))$,
\[\rd(\bm{r})^k (T_{\bm{m},\bm{n}}^{i,j})^k=T_{\bm{m},\bm{n}}^{i,j}\eta_1
+T_{\bm{m}+\bm{r},\bm{n}}^{i,j}\eta_2+T_{\bm{m},\bm{n}+\bm{r}}^{i,j}\eta_3+k!(\lambda_{\bm{m},\bm{n},\bm{r}}^{i,j})^k(\al_j^\vee(\bm{m}+\bm{n}+\bm{r}))^k.\]
This together with Lemma \ref{unprove5} gives
\[(\lambda_{\bm{m},\bm{n},\bm{r}}^{i,j})^k(\al_j^\vee(\bm{m}+\bm{n}+\bm{r}))^k.T(\bm{0})\subset W_{k-1}(k\bm{m}+k\bm{n}+k\bm{r}).\]
As $\al_j^\vee(\bm{m}+\bm{n}+\bm{r})$ acts injectively on $T$, we obtain that
\begin{align}\label{lmabdamnr}
\lambda_{\bm{m},\bm{n},\bm{r}}^{i,j}=0,\quad \forall\ i,j\in J,\ \bm{m},\bm{n},\bm{r}\in \Z^2.
\end{align}
Using this and the same argument as that in the proof of \cite[Lemma 3.5]{GL}, one gets that $\lambda_{\bm{m},\bm{n}}^{j,j}=\lambda(\al_j^\vee)$ for all $j\in J$ and $\bm{m},\bm{n}\in \Z^2$.

It retains to consider the case $i\ne j\in J$.
Firstly, one has that \begin{align*}
\(\al_i^\vee(\bm{m})\lambda(\al_i^\vee)-\lambda_{\bm{0},\bm{m}}^{i,j}\al_j^\vee(\bm{m})\).v_1(\bm{0})=
\(\al_i^\vee(\bm{m})\al_i^\vee(\bm{0})-\lambda_{\bm{0},\bm{m}}^{i,j}\al_j^\vee(\bm{m})\).v_1(\bm{0})=0,
\end{align*}
for $\bm{m}\in \Z$.
Using this, we find that as operators on $W_1(\bm{0})$,
\begin{align*}
&\lambda^{i,j}_{\bm{m},\bm{n}}\al_j^\vee(\bm{m}+\bm{n})=
\al_i^\vee(\bm{m})\al_i^\vee(\bm{n})
=\lambda_{\bm{0},\bm{m}}^{i,j}\lambda_{\bm{0},\bm{n}}^{i,j}\lambda(\al_i^\vee)^{-2}
\al_j^\vee(\bm{m})\al_j^\vee(\bm{n})\\
=&\ \lambda_{\bm{0},\bm{m}}^{i,j}\lambda_{\bm{0},\bm{n}}^{i,j}\lambda(\al_i^\vee)^{-2}\lambda(\al_j^\vee)
\al_j^\vee(\bm{m}+\bm{n}),
\end{align*}
and on the other hand,
\begin{align*}
\lambda^{i,j}_{\bm{m},\bm{n}}(\lambda_{0,\bm{m}+\bm{n}}^{i,j})^{-1}\lambda(\al_i^\vee)\al_i^\vee(\bm{m}+\bm{n})
=\lambda^{i,j}_{\bm{m},\bm{n}}\al_j^\vee(\bm{m}+\bm{n})=
\lambda(\al_i^\vee)
\al_i^\vee(\bm{m}+\bm{n}),
\end{align*} for $\bm{m},\bm{n}\in \Z^2$.
The above two equalities imply that
\begin{align}\label{lambdamnij1}\lambda_{\bm{0},\bm{m}}^{i,j}\lambda_{\bm{0},\bm{n}}^{i,j}\lambda(\al_i^\vee)^{-2}\lambda(\al_j^\vee)
=\lambda^{i,j}_{\bm{m},\bm{n}}=\lambda_{0,\bm{m}+\bm{n}}^{i,j},\ \forall\ \bm{m},\bm{n}\in \Z^2.
\end{align}
By letting $\bm{m}=\bm{0}$ in \eqref{lmabdamnr}, one gets that
\begin{align}\label{lambdamnij2}\lambda_{\bm{0},\bm{n}+\bm{r}}^{i,j}=\lambda_{\bm{0},\bm{n}}^{i,j},\end{align}
for $\bm{n},\bm{r}\in \Z^2$ with $\det{\bm{r}\choose \bm{n}}\ne 0$.
For any $\bm{m}\ne \bm{0}$, take $\bm{n}\in \Z^2$ such that $\detmn\ne 0$.
Now one can conclude from \eqref{lambdamnij1} and \eqref{lambdamnij2} that  $\lambda_{\bm{0},\bm{m}}^{i,j}=\lambda(\al_i^\vee)^2\lambda(\al_j^\vee)^{-1}.$
Note that  this also holds for the case $\bm{m}=\bm{0}$ (by definition). Therefore, again by \eqref{lambdamnij1},
we have that
\[\lambda_{\bm{m},\bm{n}}^{i,j}=\lambda(\al_i^\vee)^2\lambda(\al_j^\vee)^{-1}\] for all $\bm{m},\bm{n}\in \Z^2$ and $i,j\in J$.

Note that if $j\notin J$, then $\al_j^\vee(\bm{m}).T=0$ for all $\bm{m}\in \Z^2$.
Thus, we have proved that the element $v_1(\bm{0})$ lies in the following subspace of $T$
\[\{v\in T\mid \(\lambda(\al_j^\vee)\al_i^\vee(\bm{m})\al_i^\vee(\bm{n})-(\lambda(\al_i^\vee))^2\al_j^\vee(\bm{m}+\bm{n})\).v=0,\ \forall\
\bm{m},\bm{n}\in \Z^2\}.\]
It is obvious that the above subspace is a $\CG$-submodule of $T$ and hence equals to $T$.
This finishes the proof of the lemma.
\end{proof}

Finally, let us fix a $j\in J$. Then, by Lemma \ref{compatible}, $T$ is an irreducible  $(\mathcal R\ot \C\al_j^\vee)\rtimes \wt\CS$-module and satisfies the following quasi-associatively
\[(\al_j^\vee(\bm{m})\al_j^\vee(\bm{n})-\lambda(\al_j^\vee)\al_j^\vee(\bm{m}+\bm{n})).v=0,\]
for $\bm{m},\bm{n}\in \Z^2$ and $v\in T$. By using  \cite[Theorem 4.4]{GL} (see also \cite{JL,BT}),
there exists an irreducible finite dimensional $\mathfrak{sl}_2$-module $U$ and a $\bm{\gamma}\in \C^2$
such that  as $(\mathcal R\ot \C\al_j^\vee)\rtimes \wt\CS$-modules, $T$ is isomorphic to $T_{U,\lambda,0,\bm{\gamma},\bm{\gamma}'}$.
This together with Lemma \ref{compatible} gives that,
up to a scalar action of $\dfh$, the $\CG$-module $T$ is isomorphic to $T_{U,\lambda,0,\bm{\gamma},\bm{\gamma}'}$.
Finally, it is obvious that   Proposition \ref{ubm} is implied by this assertion.

\subsection{Proof of Theorem  \ref{mainhw1}}
In this subsection we are aim to complete the proof of Theorem \ref{mainhw1}.

We first show that the irreducible highest weight $\wt\fg$-module $V(T_{U,\lambda,\bm{\gamma},\bm{\gamma}'})\in \vartheta_{\mathrm{fin}}^\times$ by giving an explicit realization of $V(T_{U,\lambda,\bm{\gamma},\bm{\gamma}'}$,
where $U,\lambda,\bm{\gamma},\bm{\gamma}'$ are as in Theorem \ref{mainhw1}.
Let $V_{\dg}(\lambda)$ be  the irreducible highest weight $\dg$-module with highest weight $\lambda$.
We define a $\wt\fg$-module structure on the space $\CR\ot U\ot V_{\dg}(\lambda)$ as follows
\begin{align*} t^{\bm{m}}\ot x. t^{\bm{n}}\ot u\ot v&= t^{\bm{m}+\bm{n}}\ot u\ot (x.v),\\
t^{\bm{m}}\rk_i. t^{\bm{n}}\ot u\ot v&=0,\\
\rd_i. t^{\bm{n}}\ot u\ot v&=(n_i+\gamma'_i)\, t^{\bm{n}}\ot u\ot v,\\
\rd_{\bm{m}}.t^{\bm{n}}\ot u\ot v&=t^{\bm{m}+\bm{n}}\ot \(\(
                             \begin{array}{cc}
                              -m_0m_1 & m_0^2\\
                               -m_1^2 & m_0m_1 \\
                             \end{array}
                           \)
+\det {\bm{m}\choose \bm{\gamma}+\bm{n}}\).u\ot v,
\end{align*}
where $x\in \dg$, $\bm{m},\bm{n}\in \Z^2$, $v\in V$, $u\in U$ and $i=1,2$.
It is easy to see that the $\wt\fg$-module $\CR\ot U\ot V_{\dg}(\lambda)$ is isomorphic to the highest weight module $V(T_{U,\lambda,\bm{\gamma},\bm{\gamma}'})$ and lies in the category $\vartheta_{\mathrm{fin}}^\times$, as desired.

Conversely we assume now that $T$ is an irreducible weight $\wt\fg^0$-module such that the associated highest weight
$\wt\fg$-module $V=V(T)\in \vartheta_{\mathrm{fin}}^\times$. Then  we need to determine the structure of $T$.
Firstly, by Lemma \ref{existlambda} and Corollary \ref{vanishinglambda},   there is a $\lambda\in \dot{P}_+\setminus\{0\}$ such that
$h.v=\lambda(h).v$ for $h\in \dfh$ and $v\in T$.

\begin{lemt}\label{lem:notexist} There does not exist an $\bm{m}\in \Z^2\setminus \{\bm{0}\}$ and a non-zero vector $v_0\in T$ such that
\begin{align}\label{notexist}
 t^{k\bm{m}}\ot \al^\vee.v_0=0,\end{align}
 for all $\al\in \dot\Delta_+$ and $k>0$.
\end{lemt}
\begin{proof}Suppose that there exist  $\bm{m}\in \Z^2\setminus \{\bm{0}\}$ and  $v_0\in T$ such that \eqref{notexist} holds.
For any $\al\in \dot\Delta_+$, we view $V(T)$ as an integrable $L^e(\mathfrak{sl}_2)$-module via the subalgebra $\wt{\mathfrak{sl}}_2(\al,\bm{m},0)$.
Notice that $v_0\in V(T)_{\lambda(\al^\vee)}^+$.
Then, by using Lemma \ref{eloop1} (iii)  and \eqref{notexist}, one gets that $\lambda(\al^\vee)=0$ for all $\al\in \dot\Delta_+$, a contradiction.
\end{proof}

\begin{lemt}\label{ghw} There does not exist $\bm{m},\bm{n}\in \Z^2$ with $\detmn \ne 0$ and a non-zero vector $v_0\in T$ such that
\[t^{\bm{n}}\ot \al_i^\vee.v_0=\rd_{\bm{m}}.v_0=\rd_{\bm{n}}.v_0=0.\]
for all $i=1,\cdots,\ell$.
\end{lemt}
\begin{proof}  Notice that for any $k>0$ and $h\in \dfh$, there exists a non-zero constant $a_k$ such that
\[t^{k(\bm{m}+\bm{n})}\ot h=a_k \(\mathrm{ad}(\rd_{\bm{m}})\mathrm{ad}(\rd_{\bm{n}})\)^{k-1}\mathrm{ad}(\rd_{\bm{m}}) (t^{\bm{n}}\ot h).\]
 Then the assertion follows from this and
 Lemma \ref{lem:notexist}.

\end{proof}

By applying Proposition \ref{vanishing}, we know that
$T$ is an irreducible weight $\CG$-module with finite dimensional weight spaces.

\begin{lemt}\label{ghwub} As a weight $\CG$-module,  $T$ is uniformly bounded
and satisfies the condition \eqref{assump}.
\end{lemt}
\begin{proof} Let $T=\oplus_{\bm{m}\in \Z^2}T(\bm{m})$ be  as in \eqref{decT}.
For any $\bm{m}\in \Z^2$, we have the linear maps
\begin{align*}&\rd((-m_0,1)):\  T(\bm{m})\rightarrow T((0,m_1+1)),\\
& \rd((1-m_0,1)):\ T(\bm{m})\rightarrow T((1,m_1+1)),\\
&\al_i^\vee((-m_0,1)):\ T(\bm{m})\rightarrow T((0,m_1+1)),\end{align*}
where $i=1,\cdots,\ell$.
Using Lemma \ref{ghw}, one gets that
\[\ker \rd((-m_0,1))\cap \ker \rd((1-m_0,1))\cap (\cap_{i=1}^\ell \ker \al_i^\vee((-m_0,1)))=0.\]
This implies that
\[\dim T(\bm{m})\le (\ell+1)\dim T((0,m_1+1))+ \dim T((1,m_1+1)),\]
 for any $\bm{m}\in \Z^2$.
Similarly, by considering the linear maps $\rd((-1,m_1)), \rd((-1,m_1+1)), \al_i^\vee((-1,m_1))$ on the space $T((j,m_1+1))$, where  $i=1,\cdots,\ell$ and $j=0,1$, one can conclude that
\begin{align*}
\dim T((j,m_1+1))\le (\ell+1)\dim T((j-1,1))+ \dim T((j-1,2)),
\end{align*}
for any $m_1\in \Z$.
This implies that the $\CG$-module $T$ is uniformly bounded.

Now it remains to show that the $\CG$-module satisfies the condition \eqref{assump}. Indeed,
if =$\al_i^\vee(\bm{m}).T=0$ for all $i=1,\cdots,\ell$ and $\bm{m}\ne \bm{0}$, then it follows from
Lemma \ref{eloop1} (iii) that $\lambda(\al_i^\vee)=0$ for all $i$, a contradiction.
\end{proof}

Combining Proposition \ref{ubm} with Lemma \ref{ghwub}, we see that there exists an irreducible  finite dimensional
$\mathfrak{sl}_2$-module $U$, $\tilde{\lambda}\in \dfh^*\setminus\{0\}$, $\tilde{\lambda}'\in \dfh^*$ and $\bm{\gamma},\bm{\gamma}'\in \C^2$  such that the $\CG$-module $T$ is isomorphic to
 $T_{U,\tilde{\lambda},\tilde{\lambda}',\bm{\gamma},\bm{\gamma}'}$. We claim that $\tilde{\lambda}=\lambda$ (and hence $\tilde{\lambda}'=0$).
 Indeed, let $\al\in \dot{\Delta}_+$ and let $\Lambda_{b}=\Lambda_b^{\al,(1,0),0}$, $b\in \N$ be as in \eqref{keylemma1}.
We notice that for any $k\in \Z$, $t_0^k\ot \al^\vee$ acts on $\C[t_0,t_0^{-1}]\ot U\subset T$ as the polynomial operator 
\[ \tilde{\lambda}(\al^\vee)t_0^k:\ t_0^m\ot u\mapsto  \tilde{\lambda}(\al^\vee)t_0^{m+k}\ot u\] for $m\in \Z$ and $u\in U$.
  Thus one has
  \[\sum_{b\in \N}\Lambda_{b}z^b =\mathrm{exp}\(-\sum_{k\ge 1}\frac{\tilde{\lambda}(\al^\vee)t_0^k}{k}z^k\)=(1-t_0z)^{\tilde\lambda(\al^\vee)}.\]
 However, it follows from Lemma \ref{eloop1} that $\Lambda_{\lambda(\al^\vee)}\ne 0$ and $\Lambda_{b}=0$ for $b>\lambda(\al^\vee)$,
  which forces that $\tilde{\lambda}=\lambda$.

In summary, we have shown that $T$ is isomorphic to $T_{U,\lambda,0,\bm{\gamma},\bm{\gamma}'}$ as $\CG$-modules.
This in turn implies that $U$
is isomorphic to $T_{U,\lambda,\bm{\gamma},\bm{\gamma}'}$ as $\wt\fg^0$-modules,
which finishes the proof of Theorem \ref{mainhw1}.

\section{Proof of Theorem \ref{mainhw2}}\label{section5}
This section is devoted to the proof of Theorem \ref{mainhw2}.

\subsection{Highest weight  $\wh\fg$-modules}
We denote by
\[\wh\fg=\wt\fg_+\oplus \wh\CH\oplus \wt\fg_-=(\mathcal R\ot \dg)\oplus \mathcal K\oplus \sum_{\bm{m}\in \Z^2}\C\rd_{\bm{m}}\oplus \C\rd_0,\]
which is a subalgebra of $\wt\fg$ and $\wt\fg=\wh\fg\oplus \C\rd_1$.
In this subsection we introduce the notion of highest weight module for $\wh\fg$ and determine the relation
between integrable highest weight modules for $\wh\fg$ and $\wt\fg$ (of type II).

Let $\wh\fh=\wh\fg\cap \wt\fh$.  A $\wh\fg$-module $V$ is called a weight module if it admits a weight space decomposition
relative to $\wh\fh$.
 And a weight $\wh\fg$-module $V$ is called integrable if  for any $\al\in \wt\Delta^\times$,
$\wt\fg_\al$ acts locally nilponently on $V$.
For any $\wh\fg$-module $V$, we
can construct a loop module $L(V)=\C[t_1,t_1^{-1}]\ot V$ associated to it for $\wt\fg$ with the action
given by
\[x.(t_1^n\ot v)=t_1^{m+n}\ot (x.v),\quad \rd_1.(t_1^n\ot v)=n\, t_1^n\ot v,\]
for $x\in \wh\fg_m$ (see \eqref{zgraded}), $m,n\in \Z$ and $v\in V$.

Write $\mathcal O_{\mathrm{fin}}^\times$ for the category of integrable $\wh\fg$-modules with finite dimensional
weight spaces and non-trivial $\wt\fg_c$-action.
The following is easy to be checked.
\begin{lemt}\label{vandlv} Let $V$ be a weight $\wh\fg$-module. Then $V\in \mathcal O_{\mathrm{fin}}^\times$ if and only if the $\wt\fg$-module $L(V)\in  \vartheta_{\mathrm{fin}}^\times$.
\end{lemt}

Parallel to the Definition \ref{defhw2}, we introduce the following notion.

\begin{dfnt}\label{defhhw}
A $\wh\fg$-module $V$ is called a highest weight module with highest weight $\psi\in \wh\CH^*$ if there exists a non-zero vector $v\in V$ such that
\begin{enumerate}
\item[(1)]\ $V$ is generated by $v$;
\item[(2)]\ $\wt\fg_+.v=0$;
\item[(3)]\ $h.v=\psi(h)v$ for $h\in \wh\CH$.
\end{enumerate}
\end{dfnt}

Note that the condition (3) in Definition \ref{defhhw} forces that $\psi(\rk_1)=0$.
Given such a linear functional $\psi$ on $\wh\CH$.
We let $\C v_\psi$ stand for the  one dimensional $(\widetilde{\mathfrak{g}}(\mu)_+\oplus\widehat{\mathcal{H}})$-module
defined by
\[h.v_\psi=\psi(h)v_\psi,\quad \widetilde{\mathfrak{g}}(\mu)_+.v_\psi=0,\]
for $h\in\widehat{\mathcal{H}}$. Form the induced $\widehat{\mathfrak{g}}(\mu)$-module
\[\widehat{M}(\psi)=\U(\widehat{\mathfrak{g}}(\mu))\otimes_{\U(\widetilde{\mathfrak{g}}(\mu)_+\oplus\widehat{\mathcal{H}})}\C v_\psi,\]
which has an irreducible quotient, call $\widehat{V}(\psi)$.
We remark that $\widehat{V}(\psi)$ is the unique irreducible highest weight $\wh\fg$-module with highest weight $\psi$ and that
\begin{align}\label{charhwv} \{v\in \widehat{V}(\psi)\mid \wt\fg_+.v=0\}=\C v_\psi.\end{align}

The main purpose of this subsection is to prove the following proposition.
\begin{prpt}\label{hwmeq}
Let $\psi\in \CE$ with $\dim L_\psi>1$ and $b\in \C$. Then
 the irreducible type II highest weight $\wt\fg$-module $\widetilde{V}(\psi,b)\in \vartheta_{\mathrm{fin}}^\times$ 
 if and only if the irreducible highest weight
$\wh\fg$-module $\widehat{V}(\psi)\in \mathcal O_{\mathrm{fin}}^\times$.
\end{prpt}

From now on, in this subsection let $\psi$ be as in Proposition \ref{hwmeq} and let $r$ be the positive integer such that $L_{\psi}=L_r$.
Before proving Proposition \ref{hwmeq}, we need to collect some basic properties about the loop $\wt\fg$-module associated to $\wh{V}(\psi)$.

\begin{lemt}\label{basiconghhwm}
(i) The $\wt\fg$-module $L(\wh{V}(\psi))$ is a weight module and each weight lies in the set
\begin{align}\label{weightset}\{\underline{\psi}-\eta+l\delta_1\mid \eta\in \Gamma_+, l\in \Z\},\end{align}
where $\Gamma_+=\oplus_{i=0}^{\ell}\N\al_i$ and $\underline{\psi}\in \wt\fh^*$ is defined by letting $\underline{\psi}(\rd_1)=0$ and $\underline{\psi}(h)=\psi(h)$ for $h\in \wh\fh$.

(ii) $\C[t_1,t_1^{-1}]\ot \C v_{\psi}=\{v\in L(\wh{V}(\psi))\mid \wt\fg_+.v=0\}$.

(iii) As  $\wt\CH$-modules,  $\C[t_1,t_1^{-1}]\ot \C v_{\psi}$  is isomorphic to $L(\psi)$.

(iv)  For any $i\in \Z$, $\U(\wt\CH).t_1^i\ot v_{\psi}=t_1^i\C[t_1^r,t_1^{-r}]\ot \C v_{\psi}$.

(v) The $\wt\fg$-module  $L(\wh{V}(\psi))$ is generated by the elements $t_1^i\ot v_{\psi}$, $0\le i\le r-1$.
\end{lemt}
\begin{proof} The assertions (i) and (iii) are clearly, the assertion (ii) follows \eqref{charhwv} and the assertion (iv)
is implied by Lemma \ref{lem:hmod} (i). For the assertion (v), let $W$ be the $\wt\fg$-submodule
of $L(\wh{V}(\psi))$ generated by $t_1^i\ot v_{\psi}$, $i=0,\cdots,r-1$. It follows from the assertions (iii) and (iv)  that
$\C[t_1,t_1^{-1}]\ot \C v_{\psi}\subset W$. Let $t_1^m\ot u$ be a given element in $L(\wh{V}(\psi))$, where $m\in \Z$ and $u\in \wh{V}(\psi)$.
Since $\wh{V}(\psi)$ is generated by $v_{\psi}$ as $\wh\fg$-module, there exists $X\in \U(\wh\fg)$ such that
$X.v_{\psi}=u$.
Decompose $X$ into a finite summation $X=\sum_{n\in \Z} X_n$ such that $[\rd_1,X_n]=nX_n$.
Then we have
\[t_1^m\ot u=t_1^m\ot (\sum X_n.v_{\psi})=\sum t_1^m\ot (X_n.v_{\psi})=\sum X_n.(t_1^{m-n}\ot v_{\psi})\in W.\]
This proves the assertion (v).
\end{proof}

For any $i\in \Z$, denote by $L_i(\wh{V}(\psi))$ the $\wt\fg$-submodule of $L(\widehat{V}(\psi))$ generated by the
element $ t_1^i\ot v_{\psi}$.

\begin{lemt}\label{sechw1} For any $i,j\in \Z$, the elements $t_1^j\ot v_{\psi}\in L_i(\wh{V}(\psi))$ if and only if
$i\equiv j$ $(\mathrm{mod}\ r)$.
\end{lemt}
\begin{proof} For any $i\in \Z$, it follows from the fact $L_i(\wt{V}(\psi))=\U(\wt\fg_-)\U(\wt\CH).t_1^i\ot v_{\psi}$ that
\[(\C[t_1,t_1^{-1}]\ot \C v_{\psi})\cap
L_i(\wt{V}(\psi))=\U(\wt\CH).t_1^i\ot v_{\psi}.\]
Then the assertion is implied by  Lemma \ref{basiconghhwm} (iv).
\end{proof}

\begin{lemt}\label{irrcompo} For any $i\in \Z$, the $\wt\fg$-submodule of $L(\widehat{V}(\psi))$ is irreducible and is isomorphic to
$\wt{V}(\psi,i)$.
\end{lemt}
\begin{proof} Let $i\in \Z$ and let $W$ be a non-zero $\wt\fg$-submodule of $L_i(\widehat{V}(\psi))$.
Define a partial order ``$\preceq$" on the set \eqref{weightset}
by letting
$\underline{\psi}-\eta+l\delta_1 \preceq \underline{\psi}-\eta'+l'\delta_1$ if and only if $ \eta-\eta'\in \Gamma_+.$
Choose a non-zero weight vector $w\in W$ such that its associated weight is maximal with respect to the partial order ``$\preceq$".
This gives that $\wt\fg_+.w=0$, and hence $w=t_1^j\ot v_{\psi}$ for some $j\in \Z$ (see Lemma \ref{basiconghhwm} (ii)).
 By Lemma \ref{sechw1}, one has that $i\equiv j$ $(\mathrm{mod}\ r)$. This gives that $t_1^i\in \U(\wt\CH)t_1^j\subset W$, which
 proves the first assertion. The second one is obvious and is omitted.
\end{proof}

Now we are going to complete the proof of Proposition \ref{hwmeq}.
Note that one can conclude from Lemma \ref{basiconghhwm} (v) and Lemma \ref{irrcompo} that the
$\wt\fg$-module $L(\wh{V}(\psi))$ is completely reducible
\begin{align}\label{comred}
L(\wh{V}(\psi))=\oplus_{i=0}^{r-1} L_i(\wh{V}(\psi))\cong \oplus_{i=0}^{r-1} \wt{V}(\psi,i).\end{align}
Finally, one can easily check that Proposition \ref{hwmeq} is implied by \eqref{comred}, Lemma \ref{vandlv}  and
the following straightforward lemma.
\begin{lemt}
 For any $\psi\in \CE$ and $b\in\C$, one can modify the action of $\rd_1$ on the $\wt\fg$-module $\widetilde{V}(\psi,0)$
by adding the scalar action $b\,\mathrm{Id}$, and the resulting $\wt\fg$-module is isomorphic to $\widetilde{V}(\psi,b)$.
\end{lemt}

\subsection{Realization of irreducible integrable highest weight $\wh\fg$-modules}
Let $(\bm{\lambda},\bm{b},\varphi)$ be a triple as in Theorem \ref{mainhw2}.
In this  subsection we will prove that
the irreducible highest weight $\wh\fg$-module  $\widehat{V}(\psi_{\bm{\lambda},\bm{a},\varphi})\in \mathcal O_{\mathrm{fin}}^\times$.

We first recall a class of irreducible  highest weight $\wh\fg$-modules constructed in \cite{CLT}.
We set  $\dot{\mathfrak{f}}=\dg\oplus \C\rk_1\oplus\C\rd_1$, which is a finite dimensional reductive subalgebra of $\wt\fg$.
Notice that the restriction of the invariant  form $\<\cdot,\cdot\>$ on $\dot{\mathfrak{f}}$ is sitll non-degenerate.
Write
\[\bar{\mathfrak{f}}=\text{Der}(\C[t_0,t_0^{-1}])\ltimes(\C[t_0,t_0^{-1}]\otimes\dot{\mathfrak{f}})\oplus\C\rk_0\oplus\C\rk_v\]
for the Virasoro-affine Lie algebra associated to the pair $(\dot{\mathfrak{f}},\<\cdot,\cdot\>)$.
 Denote by
\[\bar{\mathfrak{f}}_0=\fh\oplus\C\rk_1\oplus\C\rd_1\oplus\C\rk_v\]
the Cartan subalgebra of $\bar{\mathfrak{f}}$.
View the affine root system $\Delta$  as a subset of $\bar{\mathfrak{f}}_0^*$ by letting
$\al(\rk_v)=\al(\rd_1)=\al(\rk_1)=0$ for $\al\in \Delta$.
Then we have the root space decomposition
$\bar{\mathfrak{f}}=\bigoplus_{\al\in \Delta}\bar{\mathfrak{f}}_\al$ and the triangular decomposition \[\bar{\mathfrak{f}}=\bar{\mathfrak{f}}_+\oplus \bar{\mathfrak{f}}_0\oplus \bar{\mathfrak{f}}_-,\]
where $\bar{\mathfrak{f}}_\al=\{x\in \bar{\mathfrak{f}}\mid [h,x]=\al(h)x,\ \forall h\in \bar{\mathfrak{f}}_0\}$
and $\bar{\mathfrak{f}}_\pm=\oplus_{\al\in \Delta_\pm}\bar{\mathfrak{f}}_\al$.
For any given linear functional $\eta$ on $\bar{\mathfrak{f}}_0$, we denote by $\C v_\eta$ the  one dimensional $(\bar{\mathfrak{f}}_+\oplus \bar{\mathfrak{f}}_0)$-module such that
 \[\bar{\mathfrak{f}}_+.v_\eta=0,\quad h.v_\eta=\eta(h)v_\eta,\quad h\in \bar{\mathfrak{f}}_0.\]
 Write $V_{\bar{\mathfrak{f}}}(\eta)$ for the unique irreducible quotient of the induced $\bar{\mathfrak{f}}$-module
\[M_{\bar{\mathfrak{f}}}(\eta)=\mathcal U(\bar{\mathfrak{f}})\otimes_{\mathcal U(\bar{\mathfrak{f}}_+\oplus \bar{\mathfrak{f}}_0)}\C v_\eta.\]

Let $(\bm{\lambda},\bm{a})\in (\fh^*)^k\times (\C^\times)^k$ be a pair which  satisfies the conditions $(c1)$ and $(c2)$.
In \cite{CLT}, we define a $\wh\fg$-module structure on the tensor product space
\[\widehat{V}(\bm{\lambda},\bm{a})=V_{\bar{\mathfrak{f}}}(\bar{\lambda}_1)\otimes\cdots\otimes V_{\bar{\mathfrak{f}}}(\bar{\lambda}_k),\]
where $\bar{\lambda}_i\in \bar{\mathfrak{f}}_0^*$ is defined as follows
\[\bar{\lambda}_i\mid_\fh=\lambda_i,\ \bar{\lambda}_i(\rk_v)=24\mu\lambda_i(\rk_0),\ \bar{\lambda}_i(\rk_1)=0=\bar{\lambda}_i(\rd_1).\]
Let $\psi_{\bm{\lambda},\bm{a}}$ be the linear functional  on  $\wh\CH$ defined by
\begin{align*}
t_1^n\ot h\mapsto \sum_{i=1}^k a_i^n \lambda_i(h),\quad
t_1^n \rk_0\mapsto \sum_{i=1}^k a_i^n \lambda_i(\rk_0),\quad
\rk_1\mapsto 0,\\
 \rd_0\mapsto \sum_{i=1}^k\lambda_i(\rd_0),\quad  t_1^m \rd_0\mapsto\sum_{i=1}^k a_i^m (\lambda_i(\rd_0)+\mu\lambda_i(\rk_0)),
\end{align*}
where  $h\in \dfh$, $n\in \Z$ and $m\in {\Z\setminus \{0\}}$. Set $v_{\bm{\lambda}}=v_{\bar{\lambda}_1}\otimes\cdots \ot v_{\bar{\lambda}_k}$.
In the following lemma we collect some  properties about the $\wh\fg$-module $\wh{V}(\bm{\lambda},\bm{a})$ proved in \cite{CLT}.
\begin{lemt}\label{reahw1}

 (i).  $\wt\fg_+.v_{\bm{\lambda}}=0$ and $h.v_{\bm{\lambda}}=\psi_{\bm{\lambda},\bm{a}}(h)$ for $h\in \wh\CH$;

  (ii).  The $\wh\fg$-module $\widehat{V}(\bm{\lambda},\bm{a})$ is irreducible;

  (iii). The $\wh\fg$-module $\widehat{V}(\bm{\lambda},\bm{a})$ lies in the category $\mathcal O_{\mathrm{fin}}^\times$.
\end{lemt}
\begin{proof}
The assertions (i) and (ii) were respectively proved in \cite[Lemma 4.1]{CLT} and  \cite[Theorem 4.2]{CLT}.
The assertion (iii) follows from the definition of the $\wh\fg$-module $\widehat{V}(\bm{\lambda},\bm{a})$ and \cite[Lemma 4.6]{CLT}.
\end{proof}

Recall that $\pi$ is the quotient map from $\wt\fg$ to $\wt\CS$.
Let \[\wt\CS=\wt\CS_+\oplus \wt\CS_0\oplus \wt\CS_-\] be the decomposition of $\wt\CS$ induced by \eqref{tridec2},
where $\wt\CS_\pm=\pi(\wt\fg_\pm)$ and $\wt\CS_0=\pi(\wt\CH)$.
We also set $\wh\CS=\pi(\wh\fg)$ and $\wh\CS_0=\pi(\wh\CH)$ so that \[\wh\CS=\wt\CS_+\oplus \wh\CS_0\oplus \wt\CS_-.\]
An $\wh\CS$-module is called a weight module if $\rd_0$ acts semi-simple.
Let $\chi:\Z\rightarrow\C$ be a function such that $\chi(0)=0$. We define a one dimensional $(\wt\CS_+\oplus\wh\CS_0)$-module
$\C v_\chi$ as follows
\[\wt\CS_+.v_\chi=0,\quad \rd_0.v_\chi=0,\quad \rd((0,m)).v_\chi=\frac{\chi(m)}{m}v_\chi,\]
for $0\neq m\in\Z$.
Write $V_{\wh\CS}(\chi)$ for the unique irreducible quotient of the following induced $\wh{\mathcal{S}}$-module
\[ \U(\wh\CS)\otimes_{\U(\wt\CS_+\oplus \wh\CS_0)}\C v_\chi.\]
It is obvious that $V_{\wh\CS}(\chi)$ is a weight $\wh\CS$-module.

The following result was proved in \cite[Theorem 3.6]{LS}, see also \cite{BGLZ}.

\begin{lemt}\label{reahw2} The $\wh{\mathcal{S}}$-module $V_{\wh{\CS}}(\chi)$ has finite dimensional weight spaces if
and only if $\chi$ is an exp-polynomial function.
\end{lemt}

The main purpose of this subsection is to prove the following result.
\begin{prpt}\label{comrela} Let $(\bm{\lambda},\bm{b},\varphi)$ be a triple which satisfies the conditions (c1)-(c3).
Then the irreducible highest weight $\wh\fg$-module  $\widehat{V}(\psi_{\bm{\lambda},\bm{a},\varphi})\in \mathcal O_{\mathrm{fin}}^\times$.
\end{prpt}
\begin{proof}
We define a function $\chi:\Z\rightarrow \C$ by letting
\[\chi(m)=m^2\(\sum_{i=1}^k a_i^m (\lambda_i(\rd_0)+\mu\lambda_i(\rk_0))\)-\varphi(m)\] for $m\in\Z$.
Notice that $\chi$ is an exp-polynomial polynomial function with $\chi(0)=0$ and that
\begin{align}\label{psilav}
\psi_{\bm{\lambda},\bm{a},\varphi}=\psi_{\bm{\lambda},\bm{a}}+\psi_\chi,\end{align} where
$\psi_\chi\in \wh\CH^*$ is defined as follows
\[\psi_\chi(t_1^n\ot h)=\psi_\chi(t_1^n\rk_0)=\psi_\chi(\rd_0)=0,\quad \psi_\chi(t_1^m\rd_0)=
-\frac{\chi(m)}{m^2},\]
for $h\in \dfh$, $m,n\in \Z$ and $m\ne 0$.

View the $\wh\CS$-module $V_{\wh\CS}(\chi)$ as a $\wh\fg$-module through the quotient map
$\wh\fg\rightarrow \wh\CS$. Form the tensor product $\wh\fg$-module $\wh{V}(\bm{\lambda},\bm{a})\ot V_{\wh\CS}(\chi)$.
It follows from Lemma \ref{reahw1} (i) and (ii) that the $\wh\fg$-module $\wh{V}(\bm{\lambda},\bm{a})$ is isomorphic to
the highest weight module $\wh{V}(\psi_{\bm{\lambda},\bm{a}})$.
Moreover,
one can easily check that as $\wh\fg$-modules, $V_{\wh\CS}(\chi)$ is isomorphic to the highest weight module
$\wh{V}(\psi_\chi)$.
Therefore, by \eqref{psilav},  $\wh{V}(\bm{\lambda},\bm{a},\chi)$  is a highest weight $\wh\fg$-module with highest weight
$\psi_{\bm{\lambda},\bm{a},\varphi}$.
Here, $\wh{V}(\bm{\lambda},\bm{a},\chi)$ stands for the $\wh\fg$-submodule  of $\wh{V}(\bm{\lambda},\bm{a})\ot V_{\wh\CS}(\chi)$
generated by the element $v_{\bm{\lambda}}\ot v_\chi$.

Now by Lemma \ref{reahw1} (iii) and Lemma \ref{reahw2}, we find that the  $\wh\fg$-module $\wh{V}(\bm{\lambda},\bm{a})\ot V_{\wh\CS}(\chi)$
lies in the category $\mathcal O_{\mathrm{fin}}^\times$ and so is its submodule $\wh{V}(\bm{\lambda},\bm{a},\chi)$.
Therefore, as the irreducible quotient of $\wh{V}(\bm{\lambda},\bm{a},\chi)$,  the $\wh\fg$-module $\wh{V}(\psi_{\bm{\lambda},\bm{a},\varphi})$ must contained in the category $\mathcal O_{\mathrm{fin}}^\times$, as required.
\end{proof}

\subsection{Proof of Theorem \ref{mainhw2}} In this subsection we complete the proof of Theorem \ref{mainhw2}.
We start with the following lemma.

\begin{lemt}\label{gands} Let $\psi\in \CE$  with $\psi(\wh\CH\cap \wt\fg_c)=0$ and $b\in \C$. Then one has that
$\wt\fg_c.\wt{V}(\psi,b)=0$ and $\wt\fg_c.\wh{V}(\psi)=0$.
\end{lemt}
\begin{proof}  Since $\wt\CS_0\cong \wt\CH/\wh\CH\cap \wt\fg_c$, there is a natural  irreducible $\wt\CS_0$-module structure on $L_{\psi,b}$. Extend it to be an $(\wt\CS_+\oplus \wt\CS_0)$-module by letting $\wt\CS_+$ acts trivially.
Write $\wt{V}_{\wt\CS}(\psi,b)$ for the irreducible quotient of the induced $\wt\CS$-module
$\U(\wt\CS)\otimes_{\U(\wt\CS_+\oplus \wt\CS_0)}L_{\psi,b}$, and view  it as a $\wt\fg$-module
through the quotient map $\pi$. Then it is easy to see that as $\wt\fg$-modules, $\wt{V}_{\wt\CS}(\psi,b)$ is isomorphic to $\wt{V}(\psi,b)$.
This proves that  $\wt\fg_c.\wt{V}(\psi,b)=0$.  The proof of the second assertion is similar to the first one and is omitted.
\end{proof}

For $i=0,1,\cdots,\ell$, we form the following subalgebras of $\wt\fg$
\[\wt{\mathfrak{sl}}_2(\al_i)=\text{Span}_{\C}\{t_1^m\otimes x_{\al_i}^+, t_1^m\otimes \al_i^\vee,
t_1^m\otimes x_{\al_i}^-,\rk_1, \rd_1\mid m\in\Z\}, \]
where $t_1^m\ot x_{\al_0}^\pm=t_0^{\pm 1}t_1^m\otimes x_{\theta}^\mp$ and $t_1^m\ot \al_0^\vee=t_1^m\rk_0-t_1^m\otimes \theta^\vee$.
It is obvious that these subalgebras  of $\wt\fg$ are isomorphic
to the affine Kac-Moody algebra of type $A_1^{(1)}$.
Write $\wh{\mathfrak{sl}}_2(\al_i)=\wt{\mathfrak{sl}}_2(\al_i)\cap \wh\fg$ for the derived subalgebra of $\wt{\mathfrak{sl}}_2(\al_i)$.

\begin{lemt}\label{greater} Let $\psi\in \CE$ and $b\in \C$. If the irreducible highest weight $\wt\fg$-module $\wt{V}(\psi,b)\in \vartheta_{\mathrm{fin}}^\times$, then $\dim L_\psi>1$.
\end{lemt}
\begin{proof} Otherwise, one has $L_\psi=\C 1$. For each $i=0,\cdots,\ell$, consider the integrable $\wt{\mathfrak{sl}}_2(\al_i)$-submodule
of $\wt{V}(\psi,b)$ generated by $1$. Note that $t_1^m\ot x_{\al_i}^+.1=0=t_1^n\ot \al_i^\vee.1$ for all $m,n\in \Z$ and $n\ne 0$.
By Lemma \ref{eloop1} (iii), one gets that $\al_i^\vee.1=0$. This together with Lemma \ref{gands} gives  $\wt\fg_c.\wt{V}(\psi,b)=0$, a contradiction.
\end{proof}

\begin{prpt}\label{charpsi}
Let $\psi\in\CE$ with $\dim L_\psi>1$. If the irreducible $\wh\fg$-module $\widehat{V}(\psi)\in\mathcal{O}_{\mathrm{fin}}^\times$, then
$\psi=\psi_{\bm{\lambda},\bm{a},\varphi}$ for some triple $(\bm{\lambda},\bm{a},\varphi)$ which satisfies the conditions (c1)-(c3).
\end{prpt}

\begin{proof} Let $i=0,\cdots,\ell$. Since $\rk_1.v_\psi=0$,  $\U(\wh{\mathfrak{sl}}_2(\al_i)).v_\psi$ is an integrable
$L(\mathfrak{sl}_2)$-module with finite dimensional weight spaces.
By applying Lemma \ref{eloop2} with $V=\U(\wh{\mathfrak{sl}}_2(\al_i)).v_\psi$ and $v=v_\psi$,
we know that there exist non-negative integers $p_{i1},\cdots,p_{it_i}$  and distinct non-zero complex numbers
$a_{i1},\cdots,a_{it_i}$ such that
\begin{align}\label{aijpij1}t_1^m\ot \al_i^\vee.v_\psi=(\sum_{1\le j\le t_i} p_{ij}a_{ij}^m)v_\psi,\end{align}
for $m\in \Z$. Notice that one can conclude from \eqref{aijpij1} and Lemma \ref{gands} that the set
\begin{align}\label{aijpij2}\{a_{ij}\mid 0\le i\le \ell, 1\le j\le t_i, p_{ij}>0\}\end{align}
is non-empty.

 Let $a_1,\cdots,a_k$ be the distinct elements in the set \eqref{aijpij2} and let $\lambda_1,\cdots,\lambda_k$ be
 the elements in $P_+\setminus\{0\}$ defined by $\lambda_s(\rd_0)=\delta_{1,s}\psi(\rd_0)$ and
 \begin{equation*}
 \lambda_s(\al_i^\vee)=\begin{cases} p_{ij},\ \te{if}\ a_s=a_{ij}\ \te{for some }j,\\
 0,\quad \te{otherwise},
 \end{cases}
 \end{equation*}
 for $s=1,\cdots,k$ and $i=0,\cdots,\ell$.
 Then by \eqref{aijpij1} one has that
 \begin{align}\label{charpsi1} \psi(t_1^m\ot \al_i^\vee)=\sum_{1\le s\le k} \lambda_s(\al_i^\vee)a_{s}^m,\quad \psi(\rd_0)=\sum_{1\le s\le k} \lambda_s(\rd_0),\end{align}
for $m\in \Z$ and $i=0,\cdots,\ell$.

Now it remains to prove that the map
\begin{align}\label{charpsi2} \varphi:\Z\rightarrow \C,\quad \varphi(m)=-m\psi(\rd_{(0,m)}),\quad m\in \Z,
\end{align}
 is an exp-polynomial function.
For convenience, we introduce the following linear map 
\[\Psi: \CR\rightarrow \U(\wh\fg),\quad t^{\bm{m}}\mapsto \rd_{\bm{m}},\,\bm{m}\in \Z^2.\]
For each $\eta\in \fh^*$, view it as an element $\wh\fh^*$ by letting $\eta(\rk_1)=0$.
 Then  $\rd_{(-1,i)}.v_\psi\in \wh{V}(\psi)_{\underline{\lambda}-\delta_0}$ for $i\in\Z$, where $\underline{\lambda}=\sum_{1\le s\le k}\lambda_s$ and $\wh{V}(\psi)_{\underline{\lambda}-\delta_0}$ denotes the weight space of $\wh{V}(\psi)$ with weight $\underline{\lambda}-\delta_0$.
Since $\dim \wh{V}(\psi)_{\underline{\lambda}-\delta_0}<\infty$, we know that there exists an $m\in\Z$
and a nonzero polynomial $P(t_1)=\sum_{i=0}^{n}p_it_1^i\in\C[t_1]$ with $p_0p_n\neq 0$ such that
\[\Psi(t_{0}^{-1}t_1^m P(t_1)).v_\psi=0.\]
For any $s\in\Z$, we have
\begin{align*}
&0=\Psi(t_0t_1^s).\Psi(t_{0}^{-1}t_1^mP(t_1)).v_\psi=[\rd_{(1,s)},\sum_{i=0}^n p_i\rd_{(-1,l+i)}].v_\psi\\
&=\sum_{i=0}^{n}p_i\{(m+i+s)\rd_{(0,m+i+s)}+\mu(m+i+s)^2t_1^{m+i+s}\rk_0\}.v_\psi.\end{align*}
This gives
\begin{align}\label{exp-poly1}
\psi\big(\sum_{i=0}^np_i\{(m+i)\rd_{(0,m+i)}+\mu(m+i)^2t_1^{m+i}\rk_0\}\big)=0,\end{align}
for all $m\in\Z$.

Define a function $\varphi':\Z\rightarrow \C$ by letting
\[\varphi'(m)=m\psi(\rd_{(0,m)})+\mu m^2\psi(t_1^m\rk_0)\]
for $m\in \Z$. Then the equation \eqref{exp-poly1} can be rewritten as follows
\begin{align}\label{charpsi3} \sum_{i=0}^n p_i\varphi'(m+i)=0.
\end{align}
It is well-known that the equality \eqref{charpsi3} implies $\varphi'$ is an exp-polynomial function (see \cite[Remark 3.6]{LS} for example).
Now, by \eqref{charpsi1} and \eqref{charpsi2}, one gets that
\[\varphi(m)=-\varphi'(m)+\mu m^2 \sum_{s=1}^k \lambda(\rk_0) a_i^m\]
for $m\in \Z$.
This proves that $\varphi$ is indeed an exp-polynomial function. \end{proof}

Finally, it is clear that Theorem \ref{mainhw2} follows from Lemma \ref{greater}, Proposition \ref{hwmeq}, Proposition \ref{comrela} and
Proposition \ref{charpsi}.

\end{document}